\documentclass[11pt]{amsart}

\usepackage{epigamath}

\usepackage[english]{babel}

\numberwithin{equation}{section}

\usepackage[shortlabels]{enumitem}

\usepackage{amsthm}
\usepackage{mathtools}
\usepackage{bbm} 
\usepackage{amsfonts}
\usepackage{fancyhdr} 
\usepackage{mathrsfs}  

\usepackage{tikz-cd}
\usepackage[all]{xy}
\usepackage{float}

\newtheorem{thm}{Theorem}[section]
\newtheorem{lem}[thm]{Lemma}
\newtheorem{cor}[thm]{Corollary}
\newtheorem{prop}[thm]{Proposition}

\theoremstyle{remark}

\newtheorem{rmk}[thm]{Remark}

\theoremstyle{definition}
\newtheorem{defi}[thm]{Definition}

\newtheorem{defithm}[thm]{Definition-Theorem}

\setlist[enumerate,1]{label={\rm(\arabic*)}, ref={\rm\arabic*}}
\newcommand{\lra}{\longrightarrow}

\DeclareMathOperator{\Aut}{Aut}
\DeclareMathOperator{\Bl}{Bl}
\DeclareMathOperator{\diag}{diag}
\DeclareMathOperator{\FL}{FL}
\DeclareMathOperator{\Fut}{Fut}
\DeclareMathOperator{\GL}{GL}
\DeclareMathOperator{\PGL}{PGL}
\DeclareMathOperator{\Hom}{Hom}
\DeclareMathOperator{\ord}{ord}
\DeclareMathOperator{\Kss}{Kss}
\DeclareMathOperator{\Pic}{Pic}
\DeclareMathOperator{\Proj}{Proj}
\DeclareMathOperator{\Spec}{Spec}
\DeclareMathOperator{\Supp}{Supp}
\DeclareMathOperator{\vol}{vol}

\def\min{{\operatorname{min}}}
\def\max{{\operatorname{max}}}
\def\inf{{\operatorname{inf}}}
\def\sup{{\operatorname{sup}}}
\def\lim{{\mathrm{lim}}}

\def\id{{\mathrm{id}}}

\def\wt{{\mathrm{wt}}}

\def\Gr{\mathrm{Gr}}
\def\CGr{{\rm CGr}}
\def\spn{\operatorname{span}}

\def\log{\operatorname{log}}
\def\Val{{\mathrm{Val}}}

\def\DH{{\mathrm{DH}}}

\def\Sym{{\mathrm{Sym}}}
\def\pr{{\mathrm{pr}}}
\def\pt{{\mathrm{pt}}}

\def\aut{\mathfrak{aut}}
\def\sl{\mathfrak{sl}}
\def\ng{\mathfrak{g}}
\def\SL{{\mathrm{SL}}}

\def\nv{\mathfrak{v}}

\newcommand{\IA}{{\mathbb A}}

\newcommand{\IC}{{\mathbb C}}

\newcommand{\IF}{{\mathbb F}}
\newcommand{\IG}{{\mathbb G}}

\newcommand{\IN}{{\mathbb N}}
 
\newcommand{\IP}{{\mathbb P}} 
\newcommand{\IQ}{{\mathbb Q}} 
\newcommand{\IR}{{\mathbb R}}

\newcommand{\IT}{{\mathbb T}}

\newcommand{\IZ}{{\mathbb Z}}

\newcommand{\CC}{{\mathcal C}}
 
\newcommand{\CE}{{\mathcal E}}
\newcommand{\CF}{{\mathcal F}}

\newcommand{\CL}{{\mathcal L}}

\newcommand{\CO}{{\mathcal O}}

\newcommand{\CX}{{\mathcal X}}

\newcommand{\tB}{{\tilde{B}}}
\newcommand{\tC}{{\tilde{C}}}
\newcommand{\tD}{{\tilde{D}}} 
\newcommand{\tE}{{\tilde{E}}}

\newcommand{\tH}{{\tilde{H}}}

\newcommand{\tL}{{\tilde{L}}}
\newcommand{\tM}{{\tilde{M}}}

\newcommand{\tQ}{{\tilde{Q}}} 
\newcommand{\tR}{{\widetilde{R}}}

\newcommand{\tT}{{\tilde{T}}}

\newcommand{\tW}{{\widetilde{W}}}

\newcommand{\tY}{{\tilde{Y}}}

\newcommand{\ttE}{{\tilde{\tilde{E}}}}

\newcommand{\ttH}{{\tilde{\tilde{H}}}}

\newcommand{\ttL}{{\tilde{\tilde{L}}}}

\newcommand{\seq}{\subseteq}
\newcommand{\lam}{\lambda}

\newcommand{\la}{\langle}
\newcommand{\ra}{\rangle}

\renewcommand{\D}{\Delta}
\newcommand{\de}{\delta}

\newcommand{\bu}{\bullet}

\newcommand{\bC}{\mathbb{C}}

\newcommand{\Bv}{\mathbf{v}}
\def\BP{\mathbf{P}}

\def\BO{\mathbf{O}}

\def\Bv{\mathbf{v}}

\def\BH{\mathbf{H}}

\def\B0{\mathbf{0}}

\newcommand{\relmiddle}[1]{\mathrel{}\middle#1\mathrel{}}
\newcommand{\GM}{\mathrm{GM}}

\newcommand{\addappendix}{%
  \appendix
  \phantomsection 
  \setcounter{section}{0}
  \section*{\appendixname}
  \addcontentsline{toc}{section}{\appendixname}
  \setcounter{section}{1}
  \setcounter{thm}{0}
}

\EpigaVolumeYear{10}{2026} \EpigaArticleNr{11} \ReceivedOn{June 9, 2024}
\InFinalFormOn{December 2, 2025}
\AcceptedOn{January 23, 2026}

\title{K-stability of special Gushel--Mukai manifolds}
\titlemark{K-stability of special Gushel--Mukai manifolds}

\author{Yuchen Liu}
\address{Department of Mathematics, Northwestern University, Evanston, IL 60208, USA}
\email{yuchenl@northwestern.edu}

\author{Linsheng Wang} 
\address{Shanghai Center for Mathematical Sciences, Fudan University, Shanghai, 200438, China}
\address{Department of Mathematics, Nanjing University, Nanjing, 210008, China}
\address{Kaiyuan International Mathematical Sciences Institute, Changsha, 410000, China}
\email{linsheng\_wang@outlook.com}

\authormark{Y.~Liu and L.~Wang}

\AbstractInEnglish{Gushel--Mukai manifolds are specific families of $n$-dimensional Fano manifolds of Picard rank $1$ and index $n-2$, where $3\leq n \leq 6$. A Gushel--Mukai $n$-fold is either ordinary, \textit{i.e.}~a hyperquadric section of a quintic Del Pezzo $(n+1)$-fold, or special, \textit{i.e.}~it admits a double cover over the quintic Del Pezzo $n$-fold branched along an ordinary Gushel--Mukai $(n-1)$-fold.
  In this paper, we prove that a general special Gushel--Mukai $n$-fold is K-stable for every $3\leq n\leq 6$. Furthermore, we give a description of the first and last walls of the K-moduli of the pair $(M,cQ)$, where $M$ is the quintic Del Pezzo fourfold (or fivefold) and $Q$ is an ordinary Gushel--Mukai threefold (or fourfold). Finally, we compute $\delta$-invariants of quintic Del Pezzo fourfolds and fivefolds that were shown to be K-unstable by K.\ Fujita and show that they admit K\"ahler--Ricci solitons.}

\MSCclass{14J45, 32Q20, 14D20}

\KeyWords{K-stability, Gushel--Mukai manifolds, quintic Del Pezzo manifolds}

\acknowledgement{The first author was partially supported by NSF CAREER Grant DMS-2237139 and the Alfred P. Sloan Foundation. The second author was partially supported by the NKRD Program of China (\#2025YFA1018100, \#2023YFA1010600, \#2020YFA0713200) and LNMS.}

\begin{document}



\maketitle

\begin{prelims}

\DisplayAbstractInEnglish

\bigskip

\DisplayKeyWords

\medskip

\DisplayMSCclass

\end{prelims}


\newpage

\setcounter{tocdepth}{1}

\tableofcontents


\section{Introduction}

The theory of K-stability, first introduced by Tian \cite{Tia97} and later reformulated algebraically by Donaldson \cite{Don02}, is an algebraic stability theory that characterizes the existence of canonical  metrics on Fano varieties. According to the Yau--Tian--Donaldson correspondence, see \cite{CDS15, Tia15, Ber16, BBJ18, LTW19, Li19, LXZ21}, a Fano variety is K-polystable if and only if it admits a K\"ahler--Einstein metric. Thus it is an important problem to check K-stability for explicit Fano manifolds. In dimension $2$, K-stability of Del Pezzo surfaces was completely determined in \cite{Tia90}. For Fano threefolds, this problem was largely solved; see \textit{e.g.} \cite{LX19, AZ21, ACC+}. In higher dimensions, much progress has been made for Fano hypersurfaces; see \textit{e.g.} \cite{Che01, Fuj19, AZ22, AZ21, Liu22}. Despite these developments, it remains a challenge to determine K-stability for most Fano manifolds in higher dimensions.

In this paper, we study the K-stability of Gushel--Mukai manifolds which are specific families of $n$\nobreakdash-dimensional Fano manifolds of Picard rank $1$ and index $n-2$, where $3\leq n \leq 6$. 
Let us recall relevant notions. Let $X$ be an $n$-dimensional Fano manifold with Picard group generated by an ample divisor $H$. The {\it index} of~$X$  is the positive integer $r$ such that $-K_X\sim rH$. The self-intersection number $d:=(H^n)$ is called the {\it degree} of $X$. 
For $r\ge n$, Kobayashi and Ochiai proved in \cite{KO73} that $X$ is isomorphic to either 
$\IP^n$ (where $r = n+1$) or a smooth quadric hypersurface in $\IP^{n+1}$ (where $r=n$). 
Fano manifolds with index $r=n-1$, known as {\it Del Pezzo manifolds}, were classified by T.~Fujita and Iskovskikh from 1977 to 1988; see for example \cite{Fuj90}. 
Fano manifolds with index $r=n-2$ were studied by Gushel \cite{Gus83} and Mukai \cite{Muk89}. It was proved by Mukai that there are only $10$ possibilities for the degree $d$. For the case $d=10$, the corresponding Fano manifolds are called {\it Gushel--Mukai} ({\it GM\,}) {\it manifolds}. See \cite{Deb20} for a comprehensive survey on Gushel--Mukai manifolds. 

There are two types of GM manifolds: {\it ordinary GM $n$-folds} $Q_n$ are quadratic sections of quintic Del Pezzo $(n+1)$-folds $M_{n+1}$ (which are smooth $(n+1)$-dimensional linear sections of $\Gr(2,5)$); {\it special GM $n$-folds} $X_n$ are double covers of quintic Del Pezzo $n$-folds $M_n$ branched along an ordinary GM $(n-1)$-fold $Q_{n-1}$. For $3\le n\le 6$, GM $n$-folds are Fano manifolds. It was proved by \cite[Theorem 5.1]{AZ21} that all the smooth GM threefolds are K-stable.

Our first main result proves the K-stability for a general member of special GM manifolds.

\begin{thm} 
\label{Theorem: Intro, general special GM are K-stable}
A general special GM manifold of dimension $4$, $5$ or $6$ is K-stable. 
\end{thm}

Since the parameter space of the special GM $n$-folds is an irreducible closed subvariety of the parameter space of the GM $n$-folds, which is also irreducible (see Section~\ref{Section: Preliminaries} for details), we have the following corollary by the openness of K-stability; see  \cite{BL18b, BLX19, LXZ21}. 

\begin{cor}
\label{Corollary: Intro, general GM are K-stable}
A general GM manifold of dimension $4$, $5$ or $6$ is K-stable. 
\end{cor}

Recall that a special GM $n$-fold $X_n$ admits a double cover $X_n\to M_n$ branched along an ordinary GM $(n-1)$-fold $Q_{n-1}$ (see Section~\ref{Section: Preliminaries}). By \cite[Theorem 1.2]{LZ20} and \cite{Zhu21}, we know that the K-polystability of $X_n$ is equivalent to the K-polystability of the log Fano pair $(M_n, \frac{1}{2}Q_{n-1})$.  Hence it is natural to first study the K-stability of quintic Del Pezzo manifolds $M_n$. 

The quintic Del Pezzo surfaces and threefolds are K-polystable, which were proved by \cite{Tia90} and \cite{CS09}, respectively. However, the quintic Del Pezzo fourfolds and fivefolds are K-unstable. They are among the first examples of K-unstable Fano manifolds with Picard number $1$ found by K.~Fujita \cite{Fuj17}. He found a $2$-plane $S\seq M_4\seq M_5$ (which is of ``non-vertex type'', see Section~\ref{Subsection: Properties of Del Pezzo fourfolds} for details) such that $\beta(E_S)<0$, where $E_S$ is the exceptional divisor of the blowup of $M_4$ or $M_5$ along $S$. We further show the following.

\begin{thm} 
\label{Theorem: Intro, delta of DP4}
The delta invariant of $M_4$ is $\delta(M_4)=\frac{25}{27}$, which is minimized by $E_S$. 
\end{thm} 

However, the delta invariant of $M_5$ is not minimized by $E_S$ (where $\frac{A_{M_5}(E_S)}{S(E_S)} = \frac{45}{46}$, see Section~\ref{Subsection: Aut of M_5} for details). There exists a $3$-space $W=\IP^3\seq M_5$ such that  $\frac{A_{M_5}(E_W)}{S(E_W)}=\frac{15}{16} < \frac{45}{46} $, where $E_W$ is the exceptional divisor of the blowup of $M_5$ along $W$. We show the following.  

\begin{thm} 
\label{Theorem: Intro, delta of DP5}
The delta invariant of $M_5$ is $\delta(M_5)=\frac{15}{16}$, which is minimized by $E_W$. 
\end{thm} 

Next, we study the wall-crossing problem of the K-moduli spaces of $(M_n, cQ_{n-1})$ for $n=4$ or $5$. Recall that for any Fano manifold $X$ and divisor $D\in |-rK_X|$, where $r>0$ is a rational number such that $rK_X$ is Cartier, the {\it K-semistable domain} of $(X,D)$ is defined by 
$$\Kss(X,D) = \left\{c\in \left[0, \min\left\{1,r^{-1}\right\}\right] \relmiddle | (X,cD) \text{ is K-semistable} \right\}. $$
We determine the K-semistable domains of $(M_4,Q_3)$ and $(M_5, Q_4)$. This is an essential step in the proof of Theorem~\ref{Theorem: Intro, general special GM are K-stable}. 

\begin{thm}
\label{Theorem: Intro, K-ss domain of DP4 and DP5}
For general $Q_3\in|\CO_{M_4}(2)|$ and $Q_4\in |\CO_{M_5}(2)|$, we have
$$\Kss(M_4,Q_3)=\left[\frac{1}{9}, \frac{7}{8}\right], \quad
\Kss(M_5,Q_4)=\left[\frac{1}{8}, \frac{4}{5}\right]. $$ 
\end{thm}

The last walls are obtained by computing beta invariants at $\ord_{Q_3}$ and $\ord_{Q_4}$. However, the first walls are more mysterious. To overcome the difficulty, we construct special quadric sections  $Q_{3,0}$ on $M_4$ and $Q_{4,0}$ on $M_5$ to obtain K-polystable pairs; see Section~\ref{Section: first and last walls} for the definition of $Q_{3,0}$ and $Q_{4,0}$. 

\begin{thm}
The log Fano pairs $(M_4,\frac{1}{9}Q_{3,0})$ and $(M_5, \frac{1}{8}Q_{4,0})$ are K-polystable. 
\end{thm}

With a similar strategy, we may consider the weighted K-stability instead of the K-stability of log Fano pairs. For the soliton candidates $\xi_0$ on $M_4$ and $\eta_0$ on $M_5$ (see Section~\ref{Section: solitons on DP4 and DP5}), we have  the following.  

\begin{thm}
\label{Theorem: Intro, weighted K-polystability}
The Fano pairs $(M_4, \xi_0)$ and $(M_5, \eta_0)$ are weighted K-polystable. 
\end{thm}

Hence by \cite{HL23,BLXZ23}, we get the existence of K\"ahler--Ricci solitons on quintic Del Pezzo manifolds,  which answers a question asked in \cite{Ino19,Del22}. 

\begin{cor}
\label{Corollary: Intro, solitons of DP4 and DP5}
Any quintic Del Pezzo fourfold or fivefold admits a K\"ahler--Ricci soliton. 
\end{cor}

The paper is organized as follows. In Section~\ref{Section: Preliminaries}, we recall the definitions of  quintic Del Pezzo manifolds and  GM manifolds. Next, we give a detailed description of the quintic Del Pezzo fourfolds and fivefolds, and compute their delta invariants in Sections~\ref{Section: quintic Del Pezzo fourfolds} and~\ref{Section: quintic Del Pezzo fivefolds}, respectively. We study the wall-crossing problem of the K-moduli spaces of $(M_4,cQ_3)$ and $(M_5,cQ_4)$ in Section~\ref{Section: first and last walls}. With a similar strategy, we show that $M_4$ and $M_5$ both admit K\"ahler--Ricci solitons in Section~\ref{Section: solitons on DP4 and DP5}. Finally, we prove the main theorem in Section~\ref{Section: K-stability of general special GM}, that is, general special GM manifolds are K-stable. 

\subsection*{Acknowledgments} We would like to thank Minghao Miao, Lu Qi, Fei Si and Shengxuan Zhou for helpful discussions. We also thank the anonymous referees for their careful reading and constructive comments.  The second author would like to thank Gang Tian for his support and guidance.

\section{Preliminaries}
\label{Section: Preliminaries}

Throughout this paper, we work over the field of complex numbers $\bC$. 

\subsection{Gushel--Mukai manifolds and quintic Del Pezzo manifolds} We recall the definitions of GM manifolds and quintic Del Pezzo manifolds in this subsection. 

For any $\IC$-vector space $W$, the projectivization $\IP(W)$ for $W$ parameterizes the $1$-dimensional subspaces of $W$. 
Let $V_5$ be a $\IC$-vector space of dimension $5$. Let $\Gr=\Gr(2, V_5)\seq \IP(\wedge^2V_5)=:\IP$ be the Grassmannian of $2$-dimensional subspaces in $V_5$, and $\CGr\seq \IP(\IC\oplus \wedge^2 V_5)=:\tilde{\IP}$ be the projective cone over $\Gr$ with vertex $v=\IP(\IC)$. We have the natural projection map $\pi: \tilde{\IP}\setminus \{v\} \to \IP$, which restricts to $\pi_\Gr\colon\CGr\setminus \{v\} \to \Gr$. For any $3\le n\le 6$, we define 
\begin{align*}
X_n &= \CGr \cap H_{6-n} \cap Q, \\
M_n &= \Gr \cap H'_{6-n},  \\
Q_{n-1} &= \Gr \cap H'_{6-n} \cap Q', 
\end{align*}
where $H_{6-n}\seq \tilde{\IP}$, $H'_{6-n} \seq \IP$ are linear subspaces of codimension $6-n$, and $Q\seq \tilde{\IP}$, $Q' \seq \IP$ are quadric hypersurfaces. 

It is not difficult to show that $X_n$ is a Fano manifold of dimension $n$ with $\Pic(X_n) = \la H \ra$, where $H= \CO_{\tilde{\IP}}(1)|_{X_n}$. We also have $(-K_{X_n})=(n-2) H$ and $H^n = 10$. 
On the other hand, we see that $M_n$ is a Fano manifold of dimension $n$ with $\Pic(M_n)=\la H \ra$, where $H=\CO_{\IP}(1)|_{M_n}$. We also have $(-K_{M_n})=(n-1)H$ and $H^n = 5$. 
We have the following classification results due to Mukai \cite{Muk89} and T.~Fujita \cite{Fuj81}, respectively. 

\begin{defithm}[\textit{cf.} \cite{Muk89}]
Let $X_n$ be a Fano manifold of dimension $n$ with Picard number~$1$,  index $n-2$ and degree $10$. Then $n\in\{3,4,5,6\}$, and there exist a codimension $(6-n)$ subspace $H_{6-n}\seq \tilde{\IP}$ and a quadric hypersurface $Q\seq \tilde{\IP}$ such that 
\begin{eqnarray*}
X_n = \CGr \cap H_{6-n} \cap Q. 
\end{eqnarray*}
The manifolds $X_n$ are called \emph{Gushel--Mukai $($GM\,$)$ manifolds}. There are two different types of $X_n$:  
\begin{enumerate}
\item If $v\in H_{6-n}$, then we have $H'_{6-n}=H_{6-n}\cap \IP \seq \tilde{\IP}$, and $\pi$ induces a double cover $X_n\to M_n$ branched along some $Q_{n-1}\seq H'_{6-n}$. We say that $X_n$ is  a \emph{special Gushel--Mukai manifold}.
    \item If $v \notin H_{6-n}$, then $n\le 5$ and $\pi$ maps $H_{6-n}\seq \tilde{\IP}$ isomorphically to $H'_{5-n}=\pi(H_{6-n})\seq \IP$, hence maps $X_n$ isomorphically to $Q_n=\pi(X_n)\seq H'_{5-n}$. We say that $X_n$ is an \emph{ordinary Gushel--Mukai manifold}.
\end{enumerate}
Note that if $n=6$, then every Gushel--Mukai manifold is special.   
\end{defithm}

\begin{defithm}[\textit{cf.} \cite{Fuj81}]
Let $M_n$ be a Fano manifold of dimension $n$ with Picard number~$1$, index $n-1$ and degree $5$. Then $M_n$ is a linear section of $\Gr(2, V_5)\seq \IP(\wedge^2V_5)$. 
The manifolds $M_n$ are called \emph{quintic Del Pezzo manifolds}. 
\end{defithm}

\begin{rmk} \label{Remark. structure of GM}\leavevmode
\begin{enumerate}
    \item\label{RsGM-1} If $n =1$ or $2$, then the above definition extends verbatim to  smooth GM curves and GM surfaces. By \cite{DK18}, a smooth GM curve is precisely a Clifford general smooth curve of genus $6$, and a smooth GM surface is precisely a Brill--Noether general smooth polarized K3 surface of degree $10$.
    \item\label{RsGM-2} By the discussions above, we know that for any special GM $n$-fold $X_n$, there exists a double cover $\pi\colon  X_n\to M_n$ branched along an ordinary GM $(n-1)$-fold $Q_{n-1}$. 
On the other hand, the parameter space of all the GM $n$-folds is irreducible and contains the parameter space of the special GM $n$-folds as an irreducible subvariety. 
\end{enumerate}
\end{rmk}

By \cite{Fuj81}, we know that the quintic Del Pezzo manifolds are unique. 

\begin{thm}[\textit{cf.}~\cite{Fuj81}]
For any $2\le n\le 6$, all the quintic Del Pezzo $n$-folds are isomorphic to each other. 
\end{thm}

The quintic Del Pezzo fourfolds and fivefolds are among the first examples of K-unstable Fano manifolds with Picard number~$1$ found by K.~Fujita \cite{Fuj17}. 

\begin{thm}[\textit{cf.}~\cite{Fuj17}]
The quintic Del Pezzo fourfolds and fivefolds are K-unstable. 
\end{thm}

On the other hand, by \cite{Tia90} and \cite{CS09}, we know that the quintic Del Pezzo surface $M_2$ and threefold $M_3$ are K-polystable. Since $M_6=\Gr(2,V_5)$ is a homogeneous space, it is K-polystable; see for example \cite[Theorems 1.2.5 and 1.4.7]{ACC+}. 

The following result summarizes fundamental properties of the moduli spaces of smooth GM manifolds.

\begin{thm}[\textit{cf.} \cite{DK20, Deb20}]
For any $3\leq n\leq 6$, there exists a quasi-projective irreducible coarse moduli space $\mathbf{M}_n^{\GM}$ parameterizing smooth GM $n$-folds. Moreover, $\dim \mathbf{M}_n^{\GM} = 25 - \frac{1}{2}(6-n)(5-n)$. The moduli space $\mathbf{M}_n^{\GM}$ is the union of an open subvariety $\mathbf{M}_{n, \mathrm{ord}}^{\GM}$ and an irreducible closed subvariety $\mathbf{M}_{n, \mathrm{spe}}^{\GM}$, complement to each other, that parameterize ordinary and special GM $n$-folds, respectively. Moreover, the following properties hold: 
\begin{enumerate}\setlength{\itemsep}{3pt}
    \item $\mathbf{M}_{n, \mathrm{ord}}^{\GM}=\emptyset$  if and only if $n = 6$.
    \item $\mathbf{M}_{n, \mathrm{spe}}^{\GM}\cong \mathbf{M}_{n-1, \mathrm{ord}}^{\GM}$ for $4\leq n\leq 6$.
    \item $\mathbf{M}_{3, \mathrm{spe}}^{\GM}$ is isomorphic to the moduli space $\mathbf{M}_{2, \mathrm{ord,ss}}^{\GM}$ parameterizing strongly smooth GM surfaces; see \cite[Definition 3.15]{DK18}.
\end{enumerate}
\end{thm}

\begin{lem} 
\label{Lemma. Aut(X_n) to Aut(M_n,Q_n-1) ker = Z/2 } 
Let $X_n\to M_n$ be the double cover in Remark~\ref{Remark. structure of GM}\,\eqref{RsGM-2}, and let $Q_{n-1}\seq M_n$ be the branched divisor. Then there exists a morphism $\Aut(X_n)\to \Aut(M_n,Q_{n-1})$ with kernel $\IZ/2\IZ. $
\end{lem}

\begin{proof}
  Let $L$ be the ample generator of $\Pic(X_n)\cong \IZ$. Then the morphism $\phi_{|L|}\colon  X_n \to |L|$ has image $M_n$ and is branched along some $Q_{n-1}\seq M_n$. For any $g\in \Aut(X_n)$, $g^*L$ is also the ample generator of $\Pic(X_n)$, that is, $g^*L\cong L$. It induces an action $\tilde{g}\colon |L|\to |L|$ preserving $M_n$ and $Q_{n-1}$. If $\tilde{g}=\id$, then $g$ is either $\id$ or an involution. We get a homomorphism $\Aut(X_n)\to \Aut(M_n,Q_{n-1})$, whose kernel is $\IZ/2\IZ$. 
\end{proof}

\subsection{K-stability}

We recall the basic notions in K-stability theory in this subsection. 

Let $(X,\D)$ be a log Fano pair, $l_0\in \IN$ be a positive integer such that $l_0(K_X+\D)$ is Cartier, and $R=R(X,\D)=\oplus_{m\in l_0\IN} R_m$ be the anti-canonical ring of $(X,\D)$, where $R_m=H^0(X,-m(K_X+\D))$. By choosing an admissible flag over $X$ (see \cite[Equation~(1.1)]{LM09}), we get an {\it Okounkov body} $\BO=\BO(-(K_X+\D))$ of the ample $\IQ$-divisor $-(K_X+\D)$. 
A {\it filtration} $\CF$ on $R$ is a collection of subspaces $\CF^\lam R_m \subseteq R_m$ for each $\lam \in \IR$ and $m\ge l_0\IN$ such that
\begin{itemize}
\item ({\it Decreasing}) $\CF^\lam R_m \supseteq \CF^{\lam'}R_m $ for  $\lam \le \lam'$,  
\item ({\it Left-continuous}) $\CF^\lam R_m=\CF^{\lam-\epsilon}R_m$ for $0<\epsilon \ll 1$,  
\item ({\it Bounded}) $\CF^\lam R_m = R_m$ for $\lam \ll 0$ and $\CF^\lam R_m = 0$ for $\lam \gg 0$,  
\item ({\it Multiplicative}) $\CF^\lam R_m \cdot \CF^{\lam'}R_{m'} \subseteq \CF^{\lam+\lam'}R_{m+m'}$. 
\end{itemize}
The filtration $\CF$ is called {\it linearly bounded} if there is a constant $C>0$ such that $\CF^{mC}R_m=0$ for all $m$. 
For example, any valuation $v$ on $X$ defines a filtration $\CF_v$ on $R$, 
$$\CF_v^\lam R_m := \{s\in R_m\mid v(s)\ge \lam\}. $$
If the log discrepancy satisfies $A_{X,\D}(v)<+\infty$, then $\CF_v$ is linearly bounded; see \cite{JM12}. 

\begin{defi}[Concave transforms]
\label{Definition. concave transform}
For any linearly bounded filtration $\CF$ on $R$ and $t\in \IR$, we have a graded linear subseries $\CF^{(t)} R_\bu \seq R_\bu$ defined by $(\CF^{(t)} R)_{m}=\CF^{mt}R_{m}$ (see \cite[Section 2.4]{BJ17}), which induces a descending family of Okounkov bodies $\BO^{(t)} = \BO(\CF^{(t)} R_\bu)\seq \BO$. By \cite[Section 2.5]{BJ17}, we define the {\it concave transform} $G_\CF\colon \BO\to \IR$ of $\CF$ by
\begin{eqnarray*} 
G_\CF(y)=\sup\left\{t\in \IR\relmiddle | y\in\BO^{(t)}\right\}.
\end{eqnarray*}
If $y \in \text{int}(\BO)_\IQ = \text{int}(\BO)\cap \IQ^n$, then we have  
\begin{eqnarray*} 
G_\CF(y)=\sup\left\{\frac{\ord_\CF(s)}{m}\relmiddle | \frac{\nv(s)}{m} = \gamma, \exists m\in l_0\IN,\,\exists s \in R_{m} \right\}, 
\end{eqnarray*}
where $\nv\colon  K(X)\to \IZ^n$ is the faithful valuation in the definition of the Okounkov body $\BO$. 
\end{defi}

We define the {\it $S$-invariant} of $\CF$ by
\begin{eqnarray*} 
S(-(K_X+\D);\CF) = \frac{1}{\Bv}\int_{\BO}G d\rho, 
\end{eqnarray*}
where $\Bv=\vol(\BO)$ is the volume of the Okounkov body $\BO$ and $d\rho$ is the Lebesgue measure. 
Moreover, for any valuation $v$ on $X$, we define the {\it Fujita--Li invariant} $\FL(v)$ (or $\beta$-invariant $\beta(v)$) by 
\begin{eqnarray*} 
\FL(v) = A_{X,\D}(v) - S(-(K_X+\D);\CF_v). 
\end{eqnarray*}
We also define the {\it $\delta$-invariant} of $(X,\D)$ by 
\begin{eqnarray*} 
\delta(X,\D) = \mathop{\inf}_{v}\frac{A_{X,\D}(v)}{S(-(K_X+\D);\CF_v)}, 
\end{eqnarray*}
where the infimum runs over all the non-trivial valuations $v$ over $X$. 

\begin{thm}[Fujita--Li's valuative criterion]
A log Fano pair $(X,\D)$ is K-stable $($resp.\ K-semistable$)$ if and only if\, we have $\FL(v)>0$ $($resp.\ $\FL(v)\geq 0)$ for any valuation $v$ over $X$. 
\end{thm}

We now  consider a Fano pair $(X,\D)$ with a torus $\IT=\IG_m^r$-action. Then the Cartier divisors $-m(K_X+\D)$ admit canonical $\IT$-linearization. Hence the anti-canonical ring $R_\bu=\oplus_{m\in l_0\IN}R_m$ admits a canonical weight decomposition $R_m=\oplus_{\alpha \in M} R_{m,\alpha}$, where $M\cong\IZ^r$ is the weight lattice. This weight decomposition defines the following {\it moment polytope} $\BP\seq M_\IR$ of the $\IT$-action: 
\begin{eqnarray*}
\BP   = \overline{\bigcup_{m\in l_0 \IN} \frac{1}{m} \BP_m}, \quad
\BP_m = \{\alpha \in M\mid R_{m, \alpha} \ne 0 \}. 
\end{eqnarray*}
Let $N=M^\vee$ be the co-weight lattice. Then any $\xi\in N_\IR$ defines a {\it toric valuation} $\wt_\xi$ on $X$, 
\begin{eqnarray} 
\label{Eqnarray: valuation of product TC}
\wt_\xi(f) = \la\alpha,\xi\ra + mA_{X,\D}\left(\wt_\xi\right), 
\end{eqnarray}
for any $f\in R_{m,\alpha}$ (see for example \cite[Equation~(5)]{Wang24}). For any $t\in \IR$, we may define 
\begin{eqnarray*}
\BP^{(t)} = \BP^{(t)}(\wt_\xi) = \overline{\bigcup_{m\in l_0 \IN} \frac{1}{m} \BP_m^{mt}}, \quad
\BP_m^{mt} = \{\alpha \in M\mid R_{m, \alpha} \ne 0, \la\alpha,\xi\ra \ge mt \},  
\end{eqnarray*}
similarly to $\BO^{(t)}(\CF_{\wt_\xi})$. 
By \cite[Definition 2.21]{HL20}, we may choose a faithful valuation $\nv\colon  K(X)^\times\to \IZ^n$ (where $\IZ^n$ is equipped with some total order respecting addition) on $X$ that is adapted to the torus action; that is, $\nv(f)=(\alpha,\nv^{r+1}(f),\dots, \nv^n(f))$ for any $f\in R_{m,\alpha}$. This valuation $\nv$ defines an Okounkov body $\BO\seq \IR^n$ of $R_\bu$, whose projection to the first $r$-factors $p\colon \BO\to  \IR^r$ has image $\BP$. Moreover, we have the restriction $p\colon \BO^{(t)}\to \BP^{(t)}$ for any $t\in \IR$. Hence the concave transform \smash{$G_\xi=G_{\CF_{\wt_\xi}}$} descends to $G^\BP_\xi\colon  \BP\to \IR$ (that is, \smash{$G_\xi=G^\BP_\xi\circ p$}), where 
\begin{eqnarray} 
\label{Eqnarray: G^BP_xi}
G^\BP_\xi(\alpha) = \sup\left\{t\in\IR\relmiddle | \alpha \in \BP^{(t)}\right\} = \la\alpha, \xi\ra + A_{X,\D}\left(\wt_\xi\right) 
\end{eqnarray}
is a linear function on $\BP$. 

Since $X$ admits a $\IT=\IG_m^r$-action, there exist a projective variety $Z$ of dimension $n-r$ and a $\IT$-equivariant birational map $\pi\colon  X \dashrightarrow Z\times \IT$, where the $\IT$-action on $Z$ is trivial. The function field $K(X)$ of $X$ is the fractional field of $K(Z)[M] = \oplus_{\alpha \in M} K(Z) \cdot 1^\alpha$. For any valuation $\mu$ on $Z$ and $\xi \in N_\IR$, we define the $\IT$-invariant valuation $v_{\mu, \xi}$ on $X$ by 
$$v_{\mu, \xi}(f)
= \min_\alpha\{\mu(f_\alpha)+\la \alpha, \xi \ra\}$$
for any $f=\sum_\alpha f_\alpha \cdot 1^\alpha \in K(Z)[M]$. If $\mu$ is the trivial valuation on $Z$, then $v_{\mu,\xi} = \wt_\xi$.  By \cite[Lemma~4.2]{BHJ17},  we know that any $\IT$-invariant valuation over $X$ is obtained in this way, and we get a non-canonical isomorphism 
$$\Val^\IT_X \cong \Val_Z \times N_\IR. $$ 
For any $v = v_{\mu, \xi_0}\in \Val^\IT_X$ and $\xi \in N_\IR$, we define the {\it $\xi$-twist} of $v$ by $v_\xi:= v_{\mu,\xi_0+\xi}$. One may check that the definition is independent of the choice of the birational map $X\dashrightarrow Z\times \IT$. 

We may define the function $\theta_\xi$ on $\Val^\IT_X$ by
\begin{eqnarray} 
\label{Eqnarray: A(v_xi) = A(v)+theta_xi(v)}
A_{X,\D}(v_\xi) = A_{X,\D}(v) + \theta_\xi(v). 
\end{eqnarray}
By \cite[Proposition 3.12]{Li19}, we have 
$$\Fut(\xi) = \FL(v_\xi) -\FL(v). $$
Hence if $\Fut|_N=0$, then 
\begin{eqnarray} 
\label{Eqnarray: S(v_xi) = S(v)+theta_xi(v)}
S(v_\xi) = S(v) + \theta_\xi(v). 
\end{eqnarray}

We recall the lower-bound estimate of the delta invariants developed by \cite{AZ22}. As in \cite[Definition~2.6]{AZ22}, we consider the following two cases: 
\begin{enumerate}
\item
Let $(X,\D)$ be a projective klt pair, and $F$ be a plt-type divisor over $X$ with the associated plt-type blowup $\pi\colon Y\to X$; that is, $\pi\colon  Y\to X$ is a projective birational morphism such that $Y$ is normal, $-F$ is a $\pi$-ample $\IQ$-Cartier divisor, and $(Y, \D_Y +F)$ is plt, where $\D_Y=\pi_*^{-1}\D$. 

\item
Let $(X,\D)$ be a projective subklt pair, and $F\seq X$ be a prime divisor such that $\D=\D'+bF$ ($b\in\IQ_{<1}, F\nsubseteq \Supp(\D')$) and $(X,\D'+F)$ is plt. In this case, we set $\pi=\id\colon  Y=X$ and $\D_Y=\D'$.  
\end{enumerate}
Let $\D_F$ be the different of $\D_Y$ on $F$. Let $V_\bu$ be a multi-graded linear series on $X$,  and $W_\bu$ be the refinement of $V_\bu$ by $F$. For any closed subvariety $Z\seq X$, we define 
\begin{eqnarray}
\delta_Z(X,\D;V_\bu):= \inf_v
\frac{A_{X,\D}(v)}{S(V_\bu;v)}, 
\end{eqnarray}
where the infimum runs over all the valuations $v$ on $X$ whose center contains $Z$. 

\begin{thm}[\textit{cf.}~{\cite[Theorem 3.2]{AZ22}}] \label{Theorem: Abban-Zhuang estimate} We have 
\begin{eqnarray}
\delta_Z(X,\D;V_\bu)\ge 
\min\left\{
\frac{A_{X,\D}(F)}{S(V_\bu;F)}, 
\mathop{\inf}_{\pi(Z')=Z} \delta_{Z'}(F, \D_F; W_\bu)
\right\}.
\end{eqnarray} 
\end{thm}

The following lemma will be used in the proof of the K-polystability of Del Pezzo pairs $(M,cQ_0)$. 

\begin{lem}
\label{Lemma. twist of G-invariant valuations}
Let $X$ be a normal variety with an algebraic group $G$-action, and $\IT\seq G$ be a subtorus lying in the center of\, $G$. 
Assume that there exist a $(G/\IT)$-variety $Z$ and a $G$-equivariant dominant rational map $X\dashrightarrow Z$, whose general fibers are $\IT$-orbits in $X$ $($hence it induces an isomorphism $\Val_{X}^\IT \cong \Val_Z\times N_\IR$, where $N$ is the co-weight lattice of the $\IT$-action on $X)$. 
If $v\in\Val_X$ is $G$-invariant, then there exist a $(G/\IT)$-invariant valuation $\mu\in \Val_Z$ and a $\xi\in N_\IR$ such that $v=v_{\mu,\xi}$.  
Conversely, if $\mu\in \Val_Z$ is $(G/\IT)$-invariant, then $v_{\mu,\xi}\in \Val_X^{\IT}$ is $G$-invariant for any $\xi\in N_\IR$. 
\end{lem}

\begin{proof}
The first assertion follows directly from $\Val_{X}^\IT \cong \Val_Z\times N_\IR$. 

For the second assertion, let $M=N^\vee$ be the weight lattice of the $\IT$-action on $X$. Then $K(X)$ is the fractional field of $K(Z)[M]$, and $K(Z)\seq K(X)$ is the subfield of $\IT$-invariant rational functions. 
Since $\IT$ lies in the center of $G$, we have characters $\chi_\alpha\colon  G\to \IC^*$ for $\alpha\in M$ such that $g^*1^\alpha = \chi_\alpha(g) \cdot 1^\alpha$.  
Hence for any $v\in \Val_X$ and $g\in G$, we have 
\begin{eqnarray*}
(g v)(1^\alpha) = v(g^*1^\alpha) = v(\chi_\alpha(g)\cdot 1^\alpha) = v(1^\alpha). 
\end{eqnarray*}
On the other hand, for any $\IT$-invariant $f\in K(X)$, the pull-back $g^*f$ is also $\IT$-invariant for any $g\in G$. Hence $K(Z)\seq K(X)$ is $G$-invariant. 
We have 
\begin{align*}
(g v_{\mu,\xi})(f_\alpha\cdot 1^\alpha) 
&= v_{\mu,\xi}(g^*(f_\alpha\cdot 1^\alpha)) 
\,=\, v_{\mu,\xi}((g^*f_\alpha)\cdot 1^\alpha)  \\
&= \mu(g^*f_\alpha)+ \la\alpha,\xi\ra
\,=\,\mu(f_\alpha) + \la\alpha,\xi\ra
\,=\, v_{\mu,\xi}(f_\alpha\cdot 1^\alpha) 
\end{align*}
for any $g\in G$, $\alpha\in M$ and $f_\alpha \in K(Z)$, where the fourth equality follows from the $(G/\IT)$-invariance of $\mu$. Hence $v_{\mu,\xi}$ is $G$-invariant. 
\end{proof}

\section{The quintic Del Pezzo fourfolds} \label{Section: quintic Del Pezzo fourfolds}

We give an overview of basic facts of the quintic Del Pezzo fourfolds in this section. We first compute the automorphism group by viewing the quintic Del Pezzo fourfolds as linear sections of $\Gr(2,V_5)$ via the Pl\"ucker embedding. T.~Fujita classified the quintic Del Pezzo manifolds by showing that they are birational to the projective space in an explicit way. This is the key ingredient in our computation of the delta minimizer of the Del Pezzo fourfolds.

\subsection{Properties of quintic Del Pezzo fourfolds}
\label{Subsection: Properties of Del Pezzo fourfolds}
Let $M=M_4$ be a quintic Del Pezzo fourfold, with  Pl\"ucker polarization $\CO_M(1)$. We denote by $S\seq M$ a plane of {\it non-vertex type} (introduced by \cite{Fuj81} with the property $c_2(N_{S/M})=2$, which will be formulated explicitly in Section~\ref{Subsection: Properties of Del Pezzo fourfolds}). Let $\pi\colon  \tilde{M}\to M $ be the blowup of $M$ along $S$ with exceptional divisor $E_S$. There is a unique divisor $D_9 \in |\CO_M(1)|$ with $\ord_{E_S}(D_9)=2$; let $\tilde{D}_9$ be the strict transform of $D_9$ on $\tilde{M}$. It was shown in \cite[Equation~(10.9)]{Fuj81} that there is a morphism $\tau\colon \tilde{M}\to \IP^4$ which is the blowup of $\IP^4$ along a twisted cubic curve $C_3$ in a hyperplane $H_9\seq \IP^4$ with exceptional divisor $\tilde{D}_9$, and $E_S$ is the strict transform of $H_9$. Hence $\tilde{D}_9$ and $E_S$ intersect transversally. Set $F=\tilde{D}_9\cap E_S$. Write the restrictions as $\pi_{E_S}\colon  E_S\to S$ and $\tau_{E_S}\colon  E_S\to H_9$. We will show that there is an $\Aut(M)$-invariant conic curve $C_2\seq S$. Set $\tT=\pi^{-1}(C_2)\seq E_S$. Then $T=\tau_{E_S}(\tT)\seq H_9$ is the tangent developable of $C_3$, which is a quartic surface with a simple cusp along~$C_3$. 

We compute the intersection numbers of some divisors on $E_S$. 
Since $\tau_{E_S}\colon E_S\to H_9\cong \IP^3$ is the blowup of $\IP^3$ along a twisted cubic $C_3$ with exceptional divisor $F$ and normal bundle $N_{C_3/H_9}\cong \CO_{C_3}(5)^{\oplus 2}$, the restriction $\tau_{F}\colon F\to C_3$ is a trivial $\IP^1$-bundle with $F|_F=-h_F+5l_F$, where $l_F=\tau_{F}^*\CO_{C_3}(1)$ and $h_F$ is the pull-back of $\CO_{\IP^1}(1)$ on a fiber. Note that $\CO_{H_9}(1)|_{C_3}\sim\CO_{C_3}(3)$ and $(\tau_{E_S}^*\CO_{H_9}(1))|_F \sim \tau_{F}^*(\CO_{H_9}(1)|_{C_3})$. We omit $\tau^*_{E_S}$ and have
\begin{eqnarray}\label{Intersection on E_S}
\CO_{H_9}(1)^2\cdot F = 0, 
\quad
\CO_{H_9}(1)\cdot F^2 =-3, 
\quad
F^3=-10. 
\end{eqnarray}

Next we compute the intersection numbers of some divisors on $\tT$. We shall denote by $\pi_{C_2}\colon \tT\to C_2$ the restriction of $\pi_{E_S}$ on $\tT$, which is a trivial $\IP^1$-bundle over $C_2$. Indeed, we have $E_S\cong\IP(N^\lor_{S/M})$ and $c_1(N_{S/M})=0$. Hence $N_{S/M}|_{C_2}$ is a trivial vector bundle of rank~$2$ on $S$, and $\tT\cong \IP(N^\lor_{S/M}|_{C_2})$. Let $l_\tT=\pi_{C_2}^*\CO_{C_2}(1)$ and $h_\tT$ be the restriction of $\CO_{E_S/S}(1)$. Then we have $\tT|_\tT\sim 4l_\tT$. Recall that we have $F\cdot \tT =2C$, where $C$ is a rational curve in the linear system $|h_\tT+l_\tT|$ isomorphic to $C_3$ via $\tau_{E_S}$ and isomorphic to $C_2$ via $\pi_{E_S}$. Also note  $\tau_{E_S}^*\CO_{H_9}(4) \sim \tT+2F$. We conclude that 
\begin{eqnarray}\label{Intersection on R}
\CO_{H_9}(1)|_\tT\sim h_\tT+2l_\tT, 
\quad
F|_\tT\sim 2h_\tT+2l_\tT. 
\end{eqnarray}

We shall present a quintic Del Pezzo fourfold explicitly in the projective space by writing down the defining equations for the convenience of computing the automorphism group and to study the Sarkisov link structure presented above. 
Recall that $\Gr=\Gr(2,V_5)$ is embedded in $\IP(\wedge^2V_5)$ via the Pl\"ucker embedding, and $M=M_4$ is a smooth intersection $\Gr\cap H\cap H'$, where $H$, $H'$ are hyperplanes in $\IP(\wedge^2V_5)$. Let $\{e_0,e_1,e_2,e_3, e_4\}$ be a basis of $V_5$, and $\{e_j^*\}\seq V_5^*$ be the corresponding dual basis. Set $e_{ij}=e_i\wedge e_j$, $e_{ij}^* = e_i^*\wedge e_j^*$. Any point of $\wedge^2 V_5$ can be denoted by 
$$X_0e_{01}+ X_1e_{02}+ X_2e_{03}+ X_3e_{04}+ X_4e_{12}+ X_5e_{13}+ X_6e_{14}+ X_7e_{23}+ X_8e_{24}+ X_9e_{34}, $$
where $X_i\in \IC$ and $[X_0,\ldots, X_9]$ forms projective coordinates
of $\IP(\wedge^2V_5)$. The defining equations of $\Gr\seq\IP(\wedge^2V_5)$ are 
\begin{eqnarray*}
\left\{ \begin{array}{ll}
X_0X_7-X_1X_5+X_2X_4=0, \\
X_0X_8-X_1X_6+X_3X_4=0, \\
X_0X_9-X_2X_6+X_3X_5=0, \\
X_1X_9-X_2X_8+X_3X_7=0, \\
X_4X_9-X_5X_8+X_6X_7=0.
\end{array} \right.
\end{eqnarray*}

Any hyperplane $H\seq \IP(\wedge^2V_5)$ corresponds to a linear functional $l\in \wedge^2V_5^*$, which can be viewed as a skew-symmetric $2$-form on $V_5$. We define 
$$l=2(e_{03}^*+e_{14}^*), \quad 
l'=2(e_{13}^*+e_{24}^*), $$
whose corresponding skew-symmetric matrices under the basis $\{e_i\}$ are
\begin{eqnarray*}
J =
\left( \begin{array}{ccccc}
0 & 0 & 0 & 1 & 0 \\
0 & 0 & 0 & 0 & 1 \\
0 & 0 & 0 & 0 & 0 \\
-1& 0 & 0 & 0 & 0 \\
0 &-1 & 0 & 0 & 0 
\end{array} \right), \quad
J' =
\left( \begin{array}{ccccc}
0 & 0 & 0 & 0 & 0 \\
0 & 0 & 0 & 1 & 0 \\
0 & 0 & 0 & 0 & 1 \\
0 &-1 & 0 & 0 & 0 \\
0 & 0 &-1 & 0 & 0 
\end{array} \right). 
\end{eqnarray*}
The hyperplanes defined by $l$ and $l'$ are $H=\{X_2+X_6=0\}$ and $H'=\{X_5+X_8=0\}$. One may check by the Jacobian criterion that $M=G\cap H\cap H'\seq \IP^7_{X_0,X_1,X_2,X_3,X_4,X_7,X_8,X_9}$ is smooth, hence is a quintic Del Pezzo fourfold. 
The defining equations of $M\seq \IP^7$ become 
\begin{eqnarray*}
\left\{ \begin{array}{ll}
X_0X_7+X_1X_8+X_2X_4=0, \\
X_0X_8+X_1X_2+X_3X_4=0, \\
X_0X_9+X_2^2-X_3X_8=0, \\
X_1X_9-X_2X_8+X_3X_7=0, \\
X_4X_9+X_8^2-X_2X_7=0.
\end{array} \right.
\end{eqnarray*}

It was shown in \cite[Equation~(10.9)]{Fuj81} that there are two types of planes on $M$, namely, the {\it non-vertex-type} planes $S$ and the {\it vertex-type} planes $R$. They are characterized by the second Chern classes of their normal bundles in $M$: 
$$c_2(N_{S/M})=2,\quad c_2(N_{R/M})=1. $$
In our explicit case, they are (see Theorem~\ref{Theorem. c(N_S/M_4)})
\begin{eqnarray}
\label{Eqnarray. def. S}
S=\{X_2=X_3=X_7=X_8=X_9=0\}, 
\end{eqnarray}
\begin{eqnarray}
R_\lambda =
\{\lambda^2X_0-\lambda X_1+X_4=
X_2+\lambda X_3=
X_7+\lambda^3X_3=
X_8-\lambda^2X_3=X_9=0\}, 
\end{eqnarray}
where $\lambda \in \IC \cup \{\infty\}$, that is, 
$R_\infty= \{X_0=X_2=X_3=X_8=X_9=0\}. $
It was proved by K.~Fujita in \cite{Fuj17} that $\beta(E_S)<0$; hence $M$ is K-unstable. 

With this explicit $M=M_4$, we can study the diagram $M \gets \tilde{M}\to \IP^4$ locally. 

Let us  work on the affine open subset $\{X_0=1\}$ of $M$, which is isomorphic to $\IA^4$ with coordinate $(X_1,X_4,X_2,X_3) = : (u,v,x,y)$. Hence
\begin{eqnarray*}
\left\{ \begin{array}{ll}
X_7=u(ux+vy)-vx, \\
X_8=-(ux+vy), \\
X_9=-y(ux+vy)-x^2, 
\end{array} \right.
\end{eqnarray*}
and $S=\{x=y=0\}$. So $\tM= \Bl_{(x,y)}M = \{\xi_0y=\xi_1x\} \seq M\times \IP^1_{\xi_0,\xi_1}$. We also replace $\tM$ by the affine open subset $\{\xi_1=1\}$ and set $s=\xi_0$, which is isomorphic to $\IA^4$ with coordinate $(u,v,s,y)$. The morphism $\pi$ is just
\begin{eqnarray*}
\pi\colon\tM \lra M,\quad (u,v,s,y) \longmapsto (u,v,sy,y). 
\end{eqnarray*}
Hence the exceptional divisor of $\pi\colon \tM\to M$ is $E_S= \{y=0\}$, and 
\begin{eqnarray*}
\left\{ \begin{array}{ll}
\pi^*(X_2)=sy, \\
\pi^*(X_3)=y,  \\
\pi^*(X_7)=(u(us+v)-vs)y, \\
\pi^*(X_8)=-(us+v)y, \\
\pi^*(X_9)=-(s^2+us+v)y^2.
\end{array} \right.
\end{eqnarray*}
The morphism $\tau\colon \tM\to \IP^4$ is determined by $|\CO_M(1)-E_S|=\pi^*\la X_2,X_3,X_7,X_8,X_9\ra$. Locally, 
\begin{eqnarray*}
\tau: (u,v,s,y) \longmapsto \left[s,1,u(us+v)-vs, -(us+v), -(s^2+us+v)y\right]. 
\end{eqnarray*}
The strict transform of $D_9=\{X_9=0\}$ is $\tD_9=\{s^2+us+v=0\}$, which is contracted by $\tau\colon\tM\to \IP^4$ to the twisted cubic curve $C_3\colon \IA^1\to \IP^4$, $t\mapsto [t,1,t^3,t^2,0]$. This is just the exceptional divisor of the blowup $\tau\colon \tM\to \IP^4$ of this twisted cubic curve. We also see that $\tau(E_S)=H_9\cong \IP^3$ is the projective $3$-space spanned by $C_3$. 

Consider the tangent developable $T$ of $C_3$ in $H_9 = \IP^3_{X,1,Z,Y}$, 
\begin{eqnarray*}
\left \{4Y\left(Y-X^2\right){}^2-4X\left(Y-X^2\right)(Z-XY)+(Z-XY)^2=0\right\}. 
\end{eqnarray*}
This also defines the cone over the tangent developable of $C_3$ in $\IP^4$, whose strict transform to $\tM$ is 
\begin{eqnarray*}
\tQ_0=\left\{4Y-4X(-u)+(-u)^2=0\right\}=\left\{u^2-4v=0\right\}, 
\end{eqnarray*}
where $X=s$, $Y=-(us+v)$ and $Z=u(us+v)-vs$ (hence $Z-XY=(-u)(Y-X^2)$). Hence this is also the strict transform via $\pi$ of the quadric hypersurface 
\begin{eqnarray}
\label{Eqnarray: Q_(0,3)}
Q_0 = \left\{X_1^2-4X_0X_4=0\right\} \seq M. 
\end{eqnarray}

\begin{rmk}
We will see that $Q_0$ is invariant under a maximal reductive subgroup $G$ of $\Aut(M)$. The log Fano pair $(M,cQ_0)$ is essential in our study of the K-stability of special GM fourfolds. 
\end{rmk}

We denote by $\tT=\tQ_0 \cap E_S$ the strict transform of the tangent developable $T$ of $C_3\seq H_9$ via $\tau$, and by $C_2=Q_0\cap S$ the corresponding conic curve in $S=\IP^2_{X_0,X_1,X_4}$. 
The restrictions of $\pi$ and $\tau$ to $E_S$ are 
\begin{align*}
\pi_{E_S}\colon E_{S} \lra S, &\quad
(u,v,s) \longmapsto (u,v), \\
\tau_{E_S}\colon E_{S} \lra \IP^3, &\quad
(u,v,s) \longmapsto [s,1,u(us+v)-vs, -(us+v)]. 
\end{align*}
We see that $E_S$ is a $\IP^1$-bundle over $S$ via $\pi$ and that it is the blowup of $H_9\seq\IP^3$ along the twisted cubic $C_3$ with exceptional divisor $F=\tD_9\cap E_S$. And $\pi_F\colon F\to S$ is a double cover ramified along $C_2\seq S$. 

\subsection{The automorphism group of a quintic Del Pezzo fourfold} 
\label{Subsection: Aut of M_4}
We compute the automorphism group of $M=M_4$ in our explicit setting in this subsection; see \cite{PV99}. We would like to mention that there is a detailed study of $\Aut(M)$ in \cite{PZ18}.

Note that $M=\{p\in \Gr\mid l(p)=l'(p)=0\}$ and $\Aut(\Gr)=\PGL_5(\IC)$. For any $P\in \GL_5(\IC)$, $P(M)=M$ is equivalent to $P^*\la l,l'\ra=\la l,l'\ra$. Hence
$$\Aut(M)=\left\{P\in \GL_5(\IC)\mid \left\la P^TJP, P^TJ'P\right\ra=\left\la J,J'\right\ra\right\}/\IG_m, $$
$$\aut(M)=\left\{Q\in \sl_5(\IC)\mid \left\la Q^TJ+JQ, Q^TJ'+JQ\right\ra=\left\la J,J'\right\ra\right\}.$$
For any $Q=(q_{ij})\in \aut(M)$, there exist $a,b,c,d\in \IC$ such that
\begin{eqnarray*}
\left\{ \begin{array}{l}
Q^TJ+JQ \hphantom{''}= aJ+bJ',\\
Q^TJ'+J'Q = cJ+dJ'.
\end{array}\right.
\end{eqnarray*}
We get some linear equations in $q_{ij}$ by solving which we see that there exist $e,f,g,h\in\IC$ such that
\begin{eqnarray}
\label{Matrix: Aut(M_4)}
Q =
\left( \begin{array}{ccccc}
-2d & 2b & 0 & e & f \\
c & -a-d & b & f & g \\
0 & 2c & -2a & g & h \\
0 & 0 & 0 & a+2d & -c \\
0 & 0 & 0 & -b & 2a+d 
\end{array} \right).
\end{eqnarray}
The Lie algebra $\aut(M)$ is non-reductive. 
Let $G\seq \Aut(M)$ be the reductive subgroup whose Lie algebra $\ng$ is generated by 
\begin{eqnarray} 
\label{Matrix: Aut(M_4,Q_0)}
\left( \begin{array}{ccccc}
-2d & 2b & 0 & 0 & 0 \\
c & -a-d & b & 0 & 0 \\
0 & 2c & -2a & 0 & 0 \\
0 & 0 & 0 & a+2d & -c \\
0 & 0 & 0 & -b & 2a+d 
\end{array} \right).
\end{eqnarray}
Then the quadric divisor $Q_0$ of $M$ defined by $\{X_1^2-4X_0X_4=0\}$ is $G$-invariant, and $G=\Aut(M,Q_0)$. It is  clear that $D_9=\{X_9=0\}\cap M$ and $S\seq D_9$ are $G$-invariant. We set $D_{9,Q}=Q_0\cap D_9$; then $C_2=D_{9,Q}\cap S= Q_0\cap S$ is the unique $G$-invariant curve in $S\cong\IP^2$. The point $o_9=[0,\dots,0,1]\in Q_0$ is the unique $G$-invariant point on $Q_0$. By \cite[Proposition 6.8]{PV99}, the $\Aut(M)$-orbit decomposition of $M$ is 
$$M=(M\setminus D_9) \sqcup
    (D_9\setminus S) \sqcup
    (S\setminus C_2) \sqcup C_2. $$
Hence $C_2\seq M$ is the unique minimal $\Aut(M)$-orbit. 

\begin{lem}
\label{Lemma. center of aut(M_4,Q_0)}
The center of\, $\ng$ is generated by $\diag(-2,-2,-2,3,3)$. 
\end{lem}

Choosing $d=-a$ in (\ref{Matrix: Aut(M_4,Q_0)}), we see that $G=\Aut(M_4,Q_0)$ has a subgroup $G_0 \cong \SL_2$ generated by 
\begin{eqnarray}
\label{Matrix: Aut(M_4,Q_0)-2}
P= 
\left( \begin{array}{cc}
\Sym^2 A & 0  \\
0 & (A^T)^{-1}  
\end{array} \right) 
=
\left( \begin{array}{ccccc}
a^2 & 2ab & b^2 & 0 & 0 \\
ac & ad+bc & bd & 0 & 0 \\
c^2 & 2cd & d^2 & 0 & 0 \\
0 & 0 & 0 & d & -c \\
0 & 0 & 0 & -b & a 
\end{array} \right) 
\end{eqnarray}
for $A=\begin{psmallmatrix}
a & b  \\
c & d  
\end{psmallmatrix}  \in \SL_2$. The $G_0$-action on $M$ lifts to $E_S$ since $S\seq M$ is $G_0$-invariant. Recall that the exceptional divisor $E_S \seq \IP^2_{X_0,X_1,X_4} \times \IP^3_{Y_2,Y_3,Y_7,Y_8}$ is defined by  
\begin{eqnarray}
\label{Eqnarray. E_S defining equation}
\left\{ \begin{array}{ll}
X_0Y_7+X_1Y_8+X_4Y_2=0, \\
X_0Y_8-X_1Y_2+X_4Y_3=0. 
\end{array} \right.
\end{eqnarray}
Since the $G_0$-action on $\IP^9_{X_0,\dots,X_9}$ is given by $\wedge^2 P$, we get the $G_0$-action on $E_S\seq \IP^3_{X_0,X_1,X_4} \times \IP^3_{Y_2,Y_3,Y_7,Y_8}$ by 
\begin{eqnarray}
\label{Eqnarray. matrix E_S}
\left(\left( \begin{array}{ccc}
a^2 & ab & b^2 \\
2ac & ad+bc & 2bd \\
c^2 & cd & d^2
\end{array} \right), 
\left( \begin{array}{ccccc}
ad^2+2abc & -a^2c & b^2d & -b^2c-2abd  \\
-3a^2b & a^3 & -b^3 & 3ab^2  \\
3c^2d & -c^3 & d^3 & -3cd^2  \\
-bc^2-2acd & ac^2 & -bd^2 & -ad^2-2bcd  
\end{array} \right) \right). 
\end{eqnarray}
We may directly check that the $G_0$-action on $E_S$ has orbit decomposition (see Section~\ref{Subsection: Properties of Del Pezzo fourfolds})  
\begin{eqnarray}
\label{Eqnarray. E_S orbit decomposition}
E_S= (E_S\setminus (\tT\cup F)) \sqcup (\tT\setminus C) \sqcup (F \setminus C) \sqcup C. 
\end{eqnarray}

\subsection{The delta function on the co-weight lattice}
In this subsection, we compute the delta function $\delta(\wt_\xi)=A_{M}(\wt_\xi)/S(-K_M; \wt_\xi)$ for any $\xi\in N_\IR$ as a preparation for computing $\delta(M)$, where $N=N(\IT)$ is the co-weight lattice of a maximal torus $\IT\seq \Aut(M)$. 
We first present an Okounkov body of $M_4$, which will be used in the computation of the expected vanishing order (or $S$-invariant, see \cite[Section 2.5]{BJ17}). 
Consider the affine chart $U_9=M\cap \{X_9\ne 0\}$; then 
\begin{eqnarray*}
\left\{ \begin{array}{l}
x_0 = -x_2^2+x_3x_8,\\
x_1 = x_2x_8-x_3x_7,\\
x_4 = -x_8^2+x_2x_7,
\end{array} \right. 
\end{eqnarray*}
where $x_i=X_i/X_9$ $(i=0,1,2,3,4,7,8,9)$ and $U_9\cong \IA^4_{x_2,x_3,x_7,x_8}$. The coordinate hyperplanes induce a faithful valuation $\nv\colon K(M)^\times\to \IZ^4$ as follows (with respect to the lexicographic ordering of $\IZ^4$): 
\begin{eqnarray*}
\left\{ \begin{array}{c}
\nv(x_2) = (1,0,0,0),\\
\nv(x_3) = (0,1,0,0),\\
\nv(x_7) = (0,0,1,0), \\
\nv(x_8) = (0,0,0,1), \\
\nv(x_9) = (0,0,0,0),
\end{array} \right. \quad
\left\{ \begin{array}{c}
\nv(x_0) = (0,1,0,1),\\
\nv(x_1) = (0,1,1,0),\\
\nv(x_4) = (0,0,0,2). 
\end{array} \right. 
\end{eqnarray*}
The convex polytope spanned by the above eight points in $\IR^4$ (denoted by $\BO$) has volume $\frac{5}{24}$, which is equal to $\vol(\CO(1))/4!$. Hence $\BO$ is an Okounkov body of the polarized variety $(M, \CO_M(1))$.

We denote by $P_a(t)=\exp(-tQ_a)$ and $P_d(t)=\exp(-tQ_d)\in \Aut(M)$ the one-parameter groups determined by 
\begin{eqnarray*}
Q_a =
\left( \begin{array}{ccccc}
0 & 0 & 0 & 0 & 0 \\
0 & -1 & 0 & 0 & 0 \\
0 & 0 &-2 & 0 & 0 \\
0 & 0 & 0 & 1 & 0 \\
0 & 0 & 0 & 0 & 2 
\end{array} \right), \quad
Q_d =
\left( \begin{array}{ccccc}
-2 & 0 & 0 & 0 & 0 \\
0 & -1 & 0 & 0 & 0 \\
0 & 0 & 0 & 0 & 0 \\
0 & 0 & 0 & 2 & 0 \\
0 & 0 & 0 & 0 & 1 
\end{array} \right). 
\end{eqnarray*}
Then $M$ admits a $\IT=\IG_m^2$-action generated by $P_a$ and $P_d$, whose actions on $\IP\cong \IP^7$ are 
\begin{eqnarray*}
P_a(t)\cdot[X_{0,1,2,3,4,7,8,9}] 
= [tX_0,t^{2}X_1,t^{-1}X_2,t^{-2}X_3,t^{3}X_4,tX_7,X_8,t^{-3}X_9],\\
P_d(t)\cdot[X_{0,1,2,3,4,7,8,9}]  
= [t^{3}X_0,t^{2}X_1,X_2,tX_3,tX_4,t^{-2}X_7,t^{-1}X_8,t^{-3}X_9].
\end{eqnarray*}
The weights of the $\IT$-action on $\IP^7$ are 
\begin{eqnarray*}
\eta(X_0,X_1,X_2,X_3,X_4,X_7,X_8,X_9) = (1,2,-1,-2,3,1,0;-3), \\
\zeta(X_0,X_1,X_2,X_3,X_4,X_7,X_8,X_9) = (3,2,0,1,1,-2,-1;-3). 
\end{eqnarray*}
Let $N_\IZ=\Hom(\IG_m,\IT)$ and $M_\IZ=\Hom(\IT,\IG_m)$ be the co-weight and weight lattices of the $\IT$-action. Then $N_\IZ$ is generated by $\eta$ and $\zeta$. 
We simply denote by $\wt_{(a,b)} = \wt_{a\eta+b\zeta}$ the valuation induced by $a\eta+b\zeta\in N_\IR$. Hence the delta function $\de(\wt)=A_M(\wt)/S(-K_M;\wt)$ is defined on $N_\IR$. 
The delta function is invariant up to scaling; that is, $\de(\wt_{(ka,kb)})=\de(\wt_{(a,b)})$ for any $k\in \IR_{>0}$. 

We should explain the relationship between the weights $(a\eta+b\zeta)(X_i)$ and the valuations $\wt_{(a,b)}(X_i)$. Indeed, they are different by a shift. 
Let $m_{(a,b)}=\min_{0\le i\le 9}(a\eta+b\zeta)(X_i)$ (we also denote the minimizer by $i_{(a,b)}$); then
\begin{eqnarray} \label{shifted wt & valuation}
\wt_{(a,b)}(X_i)=(a\eta+b\zeta)(X_i)-m_{(a,b)}
\end{eqnarray}
for any $0\le i\le9$, $i\ne 5,6$. 
For example, consider the one-parameter subgroup defined by $\zeta$, that is, $(a,b)=(0,1)$. We have $m_{(0,1)}=-3$. Then the valuation $\wt_\zeta=\wt_{(0,1)}$ is 
$$\wt_\zeta(X_{0,1,2,3,4,7,8,9}) 
=(6,5,3,4,4,1,2,0). $$

Next, we compute the delta function on $N_\IR$. For example, consider the valuation $\wt_{(0,1)}=\wt_\zeta$ computed above, 
which is a monomial valuation on $U_9$ with weight $(3,4,1,2)$ with respect to $x_2,x_3,x_7,x_8$, where $x_i=X_i/X_9$ $(i=2,3,7,8)$ is the coordinate chart of $U_9$. Hence by \cite[Proposition 5.1]{JM12}, the log discrepancy of $\wt_{(0,1)}$ is 
$$A_M(\wt_\zeta)=3+4+1+2=10. $$
One may show that the concave transform of $\wt_\zeta$ on $\BO$ is 
$$G_\zeta(w,x,y,z)=3w+4x+y+2z. $$
Similarly, the concave transform of $\wt_\eta=\wt_{(1,0)}$ is $$G_\eta(w,x,y,z)=2w+x+4y+3z. $$
Hence the expected vanishing order of $\wt_\zeta$ with respect to $-K_M=\CO_M(3)$ is 
$$S(-K_M; \wt_\zeta)=3S(\CO_M(1), \wt_\zeta)=\frac{3}{\vol(\BO)}\int_\BO G d\rho = \frac{48}{5}.$$
We see that the value of the delta function on $\wt_\zeta$ is $\de(\wt_\zeta)=\frac{A_M(\wt_\zeta)}{S(-K_M;\wt_\zeta)}=\frac{25}{24}>1$.

On the other hand, we consider the valuation $\wt_{(0,-1)}=\wt_{-\zeta}$, whose values on $X_i$ are
$$\wt_\zeta(X_{0,1,2,3,4,7,8,9}) 
=(0,1,3,2,2,5,4,6). $$
Note that $\wt_{\zeta}(X_i)+\wt_{-\zeta}(X_i)=6$. By Equation~(\ref{Eqnarray: G^BP_xi}), we have $G_\zeta+G_{-\zeta}=6$ on $\BO$. Hence
$$S(\CO_M(1);\wt_\zeta)+S(\CO_M(1);\wt_{-\zeta})=6,$$ 
so $S(-K_M;\wt_{-\zeta})=\frac{42}{5}$. The valuation $\wt_{-\zeta}$ is supported on the affine chart $U_0=M\cap\{X_0=1\}\cong \IA^4$ (with coordinates $x_1,x_2,x_3,x_4$) and is monomial with weight $(1,3,2,2)$. Hence 
$A_M(\wt_{-\zeta})=1+3+2+2=8$, and $\delta(\wt_{-\zeta})=\frac{20}{21}<1$. 
Symmetrically, we have $\delta(\wt_\eta)=\delta(\wt_{(1,0)})=\frac{25}{24}$ and $\delta(\wt_{-\eta})=\delta(\wt_{(-1,0)})=\frac{20}{21}$. 

For any $(a,b) \in \IR^2_{>0}$, we have
$$\wt_{a\eta+b\zeta}(X_{0,1,2,3,4,7,8,9})=(4a+6b,5a+5b,2a+3b,a+4b,6a+4b,4a+b,3a+2b,0). $$
Then $\wt_{a\eta+b\zeta}$ has the same center as $\wt_\zeta$ and is monomial with weight $(2a+3b,a+4b,4a+b,3a+2b)$ with respect to $x_2,x_3,x_7,x_8$. Hence $A_M(\wt_{a\eta+b\zeta})=10(a+b)$, and the concave transform of $\wt_{a\eta+b\zeta}$ on $\BO$ is $G_{a\eta+b\zeta}=aG_\eta+bG_\zeta$. So $S(\wt_{a\eta+b\zeta})=aS(\wt_\eta)+bS(\wt_\zeta)=\frac{48}{5}(a+b)$ and $\delta(\wt_{a\eta+b\zeta})=\frac{25}{24}$. 

If $a\ge b, b<0$, we simply assume $b=-1$. Then the minimizer of $a\eta-\zeta$ is 
\begin{eqnarray*}
i_{(a,-1)} = 
\left\{ \begin{array}{lll}
9, & 
4\le a, \\
3, &
\frac{2}{3}\le a\le 4,\\
0,  &
-1\le a\le \frac{2}{3},  \\
\end{array} \right. 
\end{eqnarray*}
In each case, the valuation $\wt_{(a,-1)}=\wt_{a\eta-\zeta}$ is monomial on $U_i, i=i_{(a,-1)}$ with respect to the coordinate chart. More precisely, the weight of the monomial valuation and the coordinate charts are
\begin{eqnarray*}
\,
\left\{ \begin{array}{lll}
(2a-3,a-4,4a-1,3a-2), &
(x_2,x_3,x_7,x_8), & 
4\le a, \\
(3a-2,4a-1,a+1,-a+4),  &
(x_0,x_1,x_2,x_9), &
\frac{2}{3}\le a\le 4,\\
(a+1,-2a+3,-3a+2,2a+2), &
(x_1,x_2,x_3,x_4),  &
-1\le a\le \frac{2}{3}. 
\end{array} \right. 
\end{eqnarray*}
It is not difficult to compute the concave transform on $\BO$:  
\begin{eqnarray*}
G_{a\eta-\zeta}=
\left\{ \begin{array}{lll}
aG_\eta -G_\zeta, & 
4\le a, \\
aG_\eta -G_\zeta + (4-a), &
\frac{2}{3}\le a\le 4,\\
aG_\eta -G_\zeta + (6-4a),  &
-1\le a\le \frac{2}{3}. 
\end{array} \right. 
\end{eqnarray*}
Hence we have the log discrepancies and the $S$-invariants, 
\begin{eqnarray*}
A_M(\wt_{a\eta-\zeta})=
\left\{ \begin{array}{lll}
10(a-1), \\
7a+2, \\
-2a+8, 
\end{array} \right. \quad
S(-K_M;\wt_{a\eta-\zeta})= \frac{3}{5} \cdot
\left\{ \begin{array}{lll}
16(a-1), & 
4\le a, \\
11a+4, &
\frac{2}{3}\le a\le 4,\\
-4a+14,  &
-1\le a\le \frac{2}{3}. 
\end{array} \right. 
\end{eqnarray*}
So the delta function on $\{a\ge-1, b=-1\}$ is 
\begin{eqnarray*}
\delta(\wt_{a\eta-\zeta})= \frac{5}{3} \cdot
\left\{ \begin{array}{lll}
5/8, & 
4\le a, \\
(7a+2)/(11a+4), &
\frac{2}{3}\le a\le 4,\\
(a-4)/(2a-7),  &
-1\le a\le \frac{2}{3}. 
\end{array} \right. 
\end{eqnarray*}
Symmetrically, we have the delta function on $\{a=-1, b\ge-1\}$. We conclude that 
\begin{eqnarray*}
\delta(\wt_{(a,b)}) = \delta(\nu_{a\eta+b\zeta}) = \frac{5}{3} \cdot
\left\{ \begin{array}{ll}
(4a+b)/(7a+2b),  &
a\le b\le -\frac{2}{3}a, \\
(2a-7b)/(4a-11b), &
-\frac{2}{3}a\le b\le -4a,\\
5/8, & 
a\ge-4b \,\&\, b\ge-4a, \\
(7a-2b)/(11a-4b), &
-\frac{2}{3}b\le a\le -4b,\\
(a+4b)/(2a+7b),  &
b\le a\le -\frac{2}{3}b. 
\end{array} \right. 
\end{eqnarray*}

Hence the delta function on $N_\IR$ is minimized by 
$$\wt_{(-1,-1)}(X_{0,1,2,3,4,7,8,9}) = 5\cdot(0,0,1,1,0,1,1,2), $$
which is the K-unstable center $\ord_{E_S}$ (up to rescaling by $5$) of $M_4$ found by \cite{Fuj17}.

\subsection{The delta minimizer of a quintic Del Pezzo fourfold}

It was proved in \cite{Fuj17} that $\beta(E_S)<0$. We have shown that the delta function is minimized by $E_S$ on $N_\IR$. In this subsection, we will further show that $E_S$ is a minimizer of $\delta(M)$. 

By the celebrated theory of \cite{LXZ21,Zhu21}, there exists an $\Aut(M)$-invariant prime divisor $E$ over $M$ computing the delta invariant. Hence $\delta(M;-K_M)=\delta_{C_2}(M;-K_M)$ since $C_2\seq M$ is the unique minimal $\Aut(M)$-orbit. 

\begin{thm} 
\label{Theorem E_S minimize delta M}
The delta invariant of $M_4$ is $\delta(M_4)=\frac{25}{27}$ and is minimized by $E_S$.  
\end{thm}

We have shown that 
$$\frac{A_M(E_S)}{S(-K_M;E_S)}=\delta(\nu_{(-1,-1)})=\frac{25}{27}.$$
Note that $-K_M\cong \CO_M(3)$. It suffices to show 
\begin{eqnarray} \label{delta C_2 lower bound}
\delta_{C_2}(M; \CO_M(1))\ge \frac{25}{9}.
\end{eqnarray}

\begin{proof}[Proof of Theorem~\ref{Theorem E_S minimize delta M}] 
Recall that $\tau_{E_S}\colon E_S\to \IP^3$ is the blowup of a twisted cubic curve $C_3$ with exceptional divisor $F$, and $\pi_{E_S}\colon E_S \to \IP^2$ is a $\IP^1$-bundle with $\tilde{T}=\pi_{E_S}^{-1}(C_2)$, which is the strict transform of the tangent developable $T\seq \IP^3$ of $C_3$. The divisors $F$ and $\tilde{T}$ in $E_S$ are tangent along a curve $C$ of order $2$. They are the only $G$-invariant subvarieties of $E_S$.

By Theorem~\ref{Theorem: Abban-Zhuang estimate}, we have
\begin{eqnarray} 
\delta_{C_2}(M;\CO_M(1))
\ge
\min\left\{
\frac{A_M(E_S)}{S(\CO(1);E_S)} = \frac{25}{9}, 
\mathop{\inf}_{\pi(Z)=C_2}\delta_Z(E_S; W_{\bu\bu}^{E_S})
\right\}. 
\end{eqnarray}
To get an estimate of $\delta_Z(E_S, W_{\bu\bu}^{E_S})$, we shall take the refinement of $W_{\bu\bu}^{E_S}$ by $\tT$ if $Z\ne C$, and by $F$ if $Z=C$.  Let $W_{\bu\bu\bu}^\tT$, $W_{\bu\bu\bu}^F$ be the refinements by $\tT, F$ respectively. Then 
\begin{eqnarray*} 
\delta_Z(E_S, W_{\bu\bu}^{E_S})
\ge \min\left\{
\frac{A_{E_S}(\tT)}{S(W_{\bu\bu}^{E_S};\tT)}, 
\delta_Z(\tT, W_{\bu\bu\bu}^{\tT})\right\} 
= \min\left\{\frac{1}{S(W_{\bu\bu}^{E_S};\tT)}, 
\frac{1}{S(W_{\bu\bu\bu}^{\tT}; Z)}\right\}, \\
\delta_C(E_S, W_{\bu\bu}^{E_S})
\ge \min\left\{
\frac{A_{E_S}(F)}{S(W_{\bu\bu}^{E_S};F)},
\delta_C(F, W_{\bu\bu\bu}^{F})\right\} 
= \min\left\{\frac{1}{S(W_{\bu\bu}^{E_S};F)}, 
\frac{1}{S(W_{\bu\bu\bu}^{F}; C)}\right\}. 
\end{eqnarray*}
We will show that 
\begin{alignat*}{2} 
S(W_{\bu\bu}^{E_S};\tT) &= \frac{1}{10}, &\quad
S(W_{\bu\bu\bu}^{\tT}; Z) &\le \frac{7}{25}, \\
S(W_{\bu\bu}^{E_S};F) &= \frac{1}{5}, &\quad 
S(W_{\bu\bu\bu}^{F}; C) &\le \frac{1}{5}, 
\end{alignat*}
which are all at most $\frac{9}{25}$. The proof will be finished.

{\em Step} 1.~ We first compute the refinement $W_{\bu\bu}^{E_S}$. We make the assumption that $u\in \IQ$, $m\in \IN$ are such that $mu\in\IN$ in the following. Recall that $\tilde{D}_9\seq \tilde{M}$ is the strict transform of $D_9=\{X_9=0\}\seq M$ and $\pi^* D_9 =\tilde{D}_9+2E_S$. The nef and big divisor $\pi^*\CO_M(1)-E_S$ determines the morphism $\tau\colon \tilde{M}\to \IP^4$ contracting $\tilde{D}_9$ to a twisted cubic curve $C_3$ in the hyperplane $H_9\seq \IP^4$. 
We have the following  decomposition:
\begin{align*}
\pi^*\CO_M(1)-uE_S
&\,\sim_\IQ\, 
\tilde{D}_9+(2-u)E_S =\,
(2-u)(\tilde{D}_9+E_S)-(1-u)\tilde{D}_9 \\
&\,\sim_\IQ\,
(2-u)\tau^*\CO_{\IP^4}(1)-(1-u)\tilde{D}_9  
\end{align*}
for $0\le u\le 2$. Hence $\pi^*\CO_M(1)-uE_S$ is an ample $\IQ$-divisor for $0<u<1$, and the fixed part of $\pi^*\CO_M(1)-uE_S$ is $(u-1)\tilde{D}_9$ for $1\le u\le2$.
Recall that $\tau_{E_S}\colon  E_S\to H_9=\IP^3$ is the blowup of $H_9$ along a twisted cubic $C_3$ with exceptional divisor $F=\tilde{D}_9|_{E_S}$. The refinement is defined as the quotient 
$W_{m,mu}^{E_S}=\CF_{E_S}^{mu} V_m/\CF_{E_S}^{mu+1}V_m, $
which is embedded in 
$H^0\left(E_S, m\left[\pi^*\CO_M(1)-uE_S\right]|_{E_S}\right)$. 
Using the Kawamata--Viehweg vanishing theorem, we conclude that 
\begin{eqnarray*}
W_{m,mu}^{E_S} = 
\left\{ \begin{array}{ll}
H^0\left(E_S, m\left[(2-u)\tau_{E_S}^*\CO_{H_9}(1)-(1-u)F\right]\right), & 0\le u\le 1,\\[.3em]
H^0\left(E_S, m(2-u)\tau_{E_S}^*\CO_{H_9}(1)\right)+
m(u-1)F,
& 1\le u\le 2.
\end{array} \right.
\end{eqnarray*}
In the second row,  we mean the subspace of 
$$H^0\left(E_S, m\left[(2-u)\tau_{E_S}^*\CO_{H_9}(1)+(u-1)F\right]\right) $$ 
consisting of sections with $\ord_F\ge m(u-1)$. 
We will simply denote this graded linear series by 
\begin{eqnarray*}
W_{(1,u)}^{E_S} = 
\left\{ \begin{array}{ll}
H^0\left(E_S, \CO_{H_9}(2-u)-(1-u)F\right), & 0\le u\le 1,\\[.3em]
H^0\left(E_S, \CO_{H_9}(2-u)\right)+
(u-1)F,
& 1\le u\le 2.
\end{array} \right.
\end{eqnarray*}

{\em Step} 2.~ Next we compute $S(W^{E_S}_{\bu\bu};\tT)$ and $S(W^{\tT}_{\bu\bu\bu};Z)$. More generally, let $F'(\ne F) \seq E_S$ be a prime divisor. Then $\tau_{E_S}(F')\seq H_9\cong \IP^3$ is a surface of degree $d>0$ and multiplicity $r\ge 0$ over $C_3$. Hence $d\ge 2r$ and $F'\sim \CO_{H_9}(d)-rF$. If $1-u-rv\ge 0$, we have
\begin{align*} 
\CF_{F'}^{(v)}W_{(1,u)}^{E_S}
&= H^0\left(E_S, \CO_{H_9}(2-u)
-(1-u)F
-vF'\right) +vF' \\
&= H^0\left(E_S, \CO_{H_9}(2-u-dv)-
(1-u-rv)F\right) +vF', 
\end{align*}
which is non-vanishing when $1-u-rv\le \frac{1}{2}(2-u-dv)$. If $1-u-rv\le 0$, we have 
\begin{align*} 
\CF_{F'}^{(v)}W_{(1,u)}^{E_S}
&= H^0\left(E_S, \CO_{H_9}(2-u)
-vF'\right)+(u-1)F+ vF', \\
&= H^0\left(E_S, \CO_{H_9}(2-u-dv)
\right)+(-1+u+rv)F+ vF', 
\end{align*}
which is non-vanishing when $2-u-dv\ge0$.  We set  
\begin{align*} 
\D_1^{F'}
&= \{(u,v)\mid 0\le 1-u-rv\le \frac{1}{2}(2-u-dv), v\ge 0 \}, \\
\D_2^{F'}
&= \{(u,v)\mid 1-u-rv\le 0, 2-u-dv\ge 0, v\ge 0 \}. 
\end{align*}
Then 
\begin{displaymath}
\vol\left(\CF_{F'}^{(v)}W_{(1,u)}^{E_S}\right) = 
\left\{ \begin{array}{ll}
(2-u-dv)^3-9(2-u-dv)(1-u-rv)^2 +10(1-u-rv)^3, & (u,v)\in \D_1^{F'},\\
(2-u-dv)^3, & (u,v)\in\D_2^{F'}.
\end{array} \right.
\end{displaymath}
We conclude that 
\begin{eqnarray} 
\label{Computation: S(W^E_S; F')}
S(W_{\bu\bu}^{E_S};F')
= \frac{4!}{3!}\frac{1}{5}
\int_{\D_1^{F'}\cup \D_2^{F'}} \vol\left(\CF_{F'}^{(v)}W_{(1,u)}^{E_S}\right) dudv 
= \frac{1}{5(d-r)}.  
\end{eqnarray}
In the case $F'=\tT$, we have $d=4$ and $r=2$. Hence $S(W_{\bu\bu}^{E_S};\tT)=\frac{1}{10}$. Recall that the restrictions of $\CO_{H_9}(1)$ and $F$ on $\tT$ are given by Equation~(\ref{Intersection on R}). 
Consequently,  the refinement
 $W_{\bu\bu\bu}^\tT$ is 
\begin{displaymath}
W_{(1,u,v)}^{\tT} = 
\left\{ \begin{array}{ll}
H^0\left(\tT, uh_\tT+(2-4v)l_\tT\right), & (u,v)\in\D_1^\tT,\\[.3em]
H^0\left(\tT, (2-u-4v)(h_\tT+2l_\tT)\right) +2(-1+u+2v)C, & (u,v)\in \D_2^\tT, 
\end{array} \right.
\end{displaymath}
where 
\begin{align*} 
\D_1^\tT
&= \{(u,v)\mid 0\le1-u-2v, u\ge 0,v\ge 0 \}, \\
\D_2^\tT
&= \{(u,v)\mid 1-u-2v\le0, 2-u-4v\ge 0, v\ge 0 \}. 
\end{align*}

We compute $S(W_{\bu\bu\bu}^{\tT}; Z)$ for $Z\ne C$. Since $\pi_\tT\colon \tT\to C_2$ is a trivial $\IP^1$-bundle, we may assume that $Z\sim ah_\tT+bl_\tT$ with $a>0,b\ge0$. The same procedure gives
\begin{displaymath}
\CF_Z^{w}W_{(1,u,v)}^{\tT} = 
\left\{ \begin{array}{ll}
H^0\left(\tT, (u-aw)h_\tT+(2-4v-bw)l_\tT\right) +wZ ,
& (u,v,w)\in \D_1^Z,\\[.3em]
H^0\left(\tT, (2-u-4v-aw)h_\tT+(2(2-u-4v)-bw)l_\tT\right) &\\
+wZ+ 2(-1+u+2v)C,
& (u,v,w)\in \D_2^Z, 
\end{array} \right.
\end{displaymath}
where
\begin{align*} 
\D_1^Z
&= \{(u,v,w)\mid v, w\ge 0,u+2v\le1, aw\le u, bw\le 2-4v \}, \\
\D_2^Z
&= \{(u,v,w)\mid v,w\ge 0,u+2v\ge1, aw\le 2-u-4v, bw\le2(2-u-4v)\}. 
\end{align*}
Hence
\begin{displaymath}
\vol\left(\CF_Z^{(w)}W_{(1,u,v)}^\tT\right) = 
\left\{ \begin{array}{ll}
2(u-aw)(2-4v-bw),
& (u,v,w)\in \D_1^Z,\\
2(2-u-4v-aw)(2(2-u-4v)-bw),
& (u,v,w)\in\D_2^Z, 
\end{array} \right.
\end{displaymath}

\begin{eqnarray}
\label{Computation: S(W^R;Z)}
S(W_{\bu\bu\bu}^{\tT}; Z)
= \frac{4!}{2!}\frac{1}{5}
\int_{\D_1^Z\cup \D_2^Z} \vol\left(\CF_Z^{(w)}W_{(1,u,v)}^{\tT}\right) dudvdw 
= 
\left\{ \begin{array}{ll}
\frac{7a-b}{25a^2},
& b < 2a,\\[.2em]
\frac{2(4a^2-10ab+9b^2)}{25b^3},
& b\ge 2a. 
\end{array} \right.
\end{eqnarray}
Since $a\ge1$, we conclude that $S(W_{\bu\bu\bu}^{\tT}; Z)\le 7/25$. 

In the case $Z=C$, one may show that $S(W_{\bu\bu\bu}^{\tT}; C)=2/5>9/25$. Hence we need to consider other refinements. 

{\em Step} 3.~ Finally, we compute $S(W_{\bu\bu}^{E_S};F)$ and $S(W_{\bu\bu\bu}^{F};C)$. If $0\le u\le 1$, we  naturally have  
\begin{eqnarray} \label{Filtration F first}
\CF_F^{v}W_{(1,u)}^{E_S}= H^0\left(E_S, \CO_{H_9}(2-u)-(1-u+v)F\right)+vF, 
\end{eqnarray}
which is non-vanishing when $0\le v\le u/2$. 
If $1\le u\le 2$ and $u-1\le v$, the filtration 
$\CF_F^{v}W_{(1,u)}^{E_S}$ is also formulated as Equation~(\ref{Filtration F first}), which is non-vanishing when $v\le u/2$. Finally, if $1\le u\le 2$ and $v\le u-1$, the filtration is trivial:
\begin{eqnarray} \label{Filtration F second}
\CF_F^{v}W_{(1,u)}^{E_S}=(u-1)F+H^0\left(E_S, \CO_{H_9}(2-u)\right). 
\end{eqnarray}
In a word, there are two polyhedral regions 
\begin{align} 
\label{Region. M_4 W^ES F (u,v)}
\D_1^F 
&= \{(u,v)\mid 0\le u\le 1, 0\le v\le u/2\} \cup 
    \{(u,v)\mid 1\le u\le 2, u-1\le v \} \\
&= \{(u,v)\mid 0\le v\le u/2, u-1\le v \}, \nonumber\\
\D_2^F
&= \{(u,v)\mid 1\le u\le 2, 0\le v\le u-1\} \nonumber
\end{align}
in the $(u,v)$-plane such that the filtration is formulated as Equation~(\ref{Filtration F first}) when $(u,v)\in \D_1$ and as Equation~(\ref{Filtration F second}) when $(u,v)\in \D_2$. In the former case, we have
\begin{align*}
\vol\left(\CF_F^{(v)}W_{(1,u)}^{E_S}\right) 
&= \left((2-u)\tau_{E_S}^*\CO_{H_9}(1)-(1-u+v)F\right)^3 \\
&= (2-u)^3 \cdot 1
    -3(2-u)^2(1-u+v) \cdot 0 \\
& \quad +3(2-u)(1-u+v)^2 \cdot (-3)
    -(1-u+v)^3 \cdot (-10)      \\
&= (2-u)^3-9(2-u)(1-u+v)^2+10(1-u+v)^3,  
\end{align*}
where the second equality follows from Equation~(\ref{Intersection on E_S}). In the latter case, 
\begin{eqnarray*}
\vol\left(\CF_F^{(v)}W_{(1,u)}^{E_S}\right) 
= \CO_{H_9}(2-u)^3
= (2-u)^3. 
\end{eqnarray*}
Recall $\vol(W_{\bu\bu}^{E_S})=\vol(V_\bu)=\vol(\CO_M(1))=5$. Hence we have 
\begin{eqnarray}
\label{Computation: S(W^E_S;F)}
S(W_{\bu\bu}^{E_S}; F)
=\frac{1}{5/4!} 
\int_{\D_1^F\cup \D_2^F} 
\frac{\vol\left(\CF_F^{(v)}W_{(1,u)}^{E_S}\right)}{3!} dudv = \frac{1}{5}. 
\end{eqnarray}

Next we take the refinement of $W^{E_S}_{\bu\bu}$ by $F$, which is 
\begin{eqnarray*}
W_{(1,u,v)}^{F} = 
H^0\left(F, (1-u+v)h_F+(1+2u-5v)l_F\right) 
\end{eqnarray*}
for $(u,v)\in \D_1^F=\{(u,v):0\le v\le u/2,u-1\le v\}$. Hence
\begin{eqnarray*}
\CF^{(w)}_C W^F_{(1,u,v)} = 
H^0\left(F,(1-u+v-w)h_F +(1+2u-5v-w)l_F\right), 
\end{eqnarray*}
which is non-vanishing when $(u,v,w)\in \D^C$,  where 
\begin{eqnarray}
\label{Region. M_4 W^F C}
\D^C=\{(u,v,w)\mid 0\le v\le u/2, u-1\le v, 0\le w\le 1-u+v \}. 
\end{eqnarray}
Hence 
\begin{eqnarray}
\label{Computation: S(W^F; C)}
S(W_{\bu\bu\bu}^{F}; C)
= \frac{1}{5/4!}
\int_{\D^C} 
\frac{\vol\left(\CF_C^{(w)}W_{(1,u,v)}^{F}\right)}{2!} dudvdw 
= \frac{1}{5}. 
\end{eqnarray}
Combining Equations~(\ref{Computation: S(W^E_S; F')}), (\ref{Computation: S(W^R;Z)}), (\ref{Computation: S(W^E_S;F)}) and (\ref{Computation: S(W^F; C)}), we have 
$$\delta_{C_2}(M_4;\CO_M(1))
\ge 
\min\left\{
\frac{25}{9}, \frac{25}{7}
\right\} 
= \frac{25}{9}. $$
The proof is finished. 
\end{proof}

\section{The quintic Del Pezzo fivefolds}
\label{Section: quintic Del Pezzo fivefolds}

We find the delta minimizer of a quintic Del Pezzo fivefold in this section. 

\subsection{The automorphism group and delta invariant of a quintic Del Pezzo fivefold}
\label{Subsection: Aut of M_5}
Let $M= \Gr(2,V_5)\cap \{X_5+X_8=0\} \seq \IP^9$, which is a smooth quintic Del Pezzo fivefold. We shall omit $X_5$ and view $M$ as a subvariety of $\IP^8_{X_0,X_1,X_2,X_3,X_4,X_6,X_7,X_8,X_9}$. With the same argument as in the previous section, one can show that the endomorphism algebra
$\aut(M)\seq \sl_5$ is generated by
\begin{eqnarray*}
\left( \begin{array}{ccccc}
-2\lam & \lam_1 & \lam_2 & \lam_3 & \lam_4 \\
0 & a & b & f & g \\
0 & c & d & g & h \\
0 & p & q & \lam-a & -c \\
0 & q & r & -b & \lam-d 
\end{array} \right). 
\end{eqnarray*}
There are only two orbits on $M$ via the $\Aut(M)$-action, namely, the closed orbit 
\begin{eqnarray}
\label{Eqnarray. def. W}
W=\{X_4=X_5=X_6=X_7=X_8=X_9=0\}\cong \IP^3 \seq \IP^9 
\end{eqnarray}
and the open orbit $M\setminus W$; see \cite[Proposition 5.3]{PV99}. Let $E_W$ be the exceptional divisor of the blowup $\pi\colon \tilde{M}\to M$ of $M$ along $W$. Then we have 
\begin{eqnarray*}
\frac{A_M(E_W)}{S(-K_M;E_W)} 
= \frac{15}{16} < \frac{45}{46} =
\frac{A_M(E_S)}{S(-K_M;E_S)}, 
\end{eqnarray*}
where the second equality follows from \cite{Fuj17}. For the first one, we now state a proof. Using Schubert calculus, one can show that $c_1(N_{W/M})=0$ and $c_2(N_{W/M})=c_1^2$ (see Theorem~\ref{Theorem. c(N_W/M_5)}), where $c_1=c_1(\CO_W(1))$. Let $h=c_1(\CO_{E_W/W}(1))$. Then we have intersection numbers
\begin{eqnarray}
\label{Formula: Intersection theory on E_W: sigma_1 and h_1}
\left( \begin{array}{ccccc}
c_1^4 & c_1^3h & c_1^2h^2 & c_1h^3 & h^4 \\
0 & 1 & 0 & -1 & 0 
\end{array} \right). 
\end{eqnarray}
Hence we have 
$\vol_{\tilde{M}|E_W}(\CO_M(1)-tE_W)
=(c_1+t h)^4 
= 4t-4t^3$, and the pseudo-effective threshold of $E_W$ with respect to $\CO_M(1)$ is equal to $1$. We have 
\begin{eqnarray}
\label{Computation: S(O(1);E_W)}
S(\CO_M(1);E_W) 
= \frac{n}{V}\int_0^1 t\cdot 
(4t-4t^3) \cdot dt 
\,=\, \frac{8}{15}, 
\end{eqnarray}
where $V= \vol(\CO_M(1)) =5$ and $n=\dim M =5$. 
We conclude by $-K_M=\CO_M(4)$ and $A_M(E_W)=2$. 
Furthermore, we have the following. 

\begin{thm}
The delta invariant of\, $M=M_5$ is $\delta(M_5) = \frac{15}{16}$ and is minimized by $E_W$. 
\end{thm}

\begin{proof}
Let $W^{E_W}_{\bu\bu}$ be the refinement of $R_\bu=R(M,\CO(1))$ by $E_W$. By equivariant K-stability and Theorem~\ref{Theorem: Abban-Zhuang estimate},  it suffices to show that 
\begin{eqnarray*}
S(W_{\bu\bu}^{E_W};Z) \le \frac{4}{15}
\end{eqnarray*}
for any prime divisor $Z\seq E_W$ dominating $W$. 
We firstly see that $\CO_M(1)-tE_W$ is ample for $0<t<1$ and is not pseudo-effective for $t>1$. Hence the $E_W$-refinement of $\CO_M(1)$ is
$$W_{(1,t)}=W^{E_W}_{(1,t)}:= (\CO_M(1)-tE_W)|_{E_W}= \CO_W(1)+\CO_{E_W/W}(t)$$ 
for $0<t<1$, which is always ample. So $\vol(W_{(1,t)})=(\sigma_1+t h)^4 = 4t-4t^3. $
Since $\pi\colon  E_W\to W$ is a $\IP^1$-bundle and $Z\seq E_W$ dominates $W$, we have $Z\sim \CO_W(a)+\CO_{E_W/W}(b)$ for $1\le b \le a$. Then we have 
\begin{equation*}
S(W_{\bu\bu}^{E_W};Z) 
= \frac{1}{5/5!}
\left(
\int_{0}^{\frac{b}{a}} \int_0^{\frac{t}{b}} 
\frac{\vol(W_{(1,t)}-sZ)}{4!} ds dt, + 
\int_{\frac{b}{a}}^1 \int_0^{\frac{1-t}{a-b}} 
\frac{\vol(W_{(1,t)}-sZ)}{4!} ds dt 
\right) 
 =  \frac{1}{6a} \, \le \frac{1}{6}. \qedhere
\end{equation*}
\end{proof}

We are also interested in the pairs $(M, Q)$ for $Q\in |\CO_M(2)|$, in particular for those $Q$ with large symmetries. They are closely related to GM varieties. 
For $M=M_4$, there is a quadric divisor $Q_0$ (defined by $X_1^2-4X_0X_4$) such that the subgroup $\Aut(M_4,Q_0)\seq \Aut(M_4)$ is reductive; see Equation~(\ref{Matrix: Aut(M_4,Q_0)}).  
For $M=M_5$, we consider the quadric divisor $Q_0$ defined by $\{X_0X_2+X_1X_3=0\}$. Then the endomorphism algebra $\ng$ of the pair $(M, Q_0)$ is generated by 
\begin{eqnarray}
\label{Matrix: Aut(M_5,Q_0)}
\left( \begin{array}{ccccc}
-2\lam & 0 & 0 & 0 & 0 \\
0 & a & b & 0 & 0 \\
0 & c & d & 0 & 0 \\
0 & 0 & 0 & \lam-a & -c \\
0 & 0 & 0 & -b & \lam-d 
\end{array} \right). 
\end{eqnarray}

\begin{lem}
\label{Lemma. center of aut(M_5,Q_0)}
The center of $\ng$ is generated by $\diag(-2,-2,-2,3,3)$ and $\diag(-4,1,1,1,1)$. 
\end{lem}

We use the same notation $Q_x$ and $P_x(t)=\exp(-tQ_x)$ as in Section~\ref{Subsection: Aut of M_4}. Note that $M$ admits an action of $\IT=\IG_m^3$ generated by $P_\lam$, $P_a$ and $P_d$, whose actions on $\IP\cong \IP^8$ are given by
\begin{align*}
P_\lam(t) \cdot [X_i] 
&= [t^{2}X_0,t^{2}X_1,t^{1}X_2,t^{2}X_3,X_4,t^{-1}X_6,t^{-1}X_7,t^{-1}X_8,t^{-2}X_9],\\
P_a(t) \cdot [X_i] 
&= [t^{-1}X_0,X_1,tX_2,X_3,t^{-1}X_4,t^{-1}X_6,tX_7,X_8,tX_9], \\
P_d(t) \cdot [X_i] 
&= [X_0,t^{-1}X_1,X_2,tX_3,t^{-1}X_4,tX_6,t^{-1}X_7,X_8,tX_9]. 
\end{align*}
We denote by $\eta, \varsigma, \zeta$ the weights of the one-parameter subgroups $P_\lam$, $P_a$, $P_d$, respectively; that is,  
\begin{align*}
\eta &= (2,2,1,1,0,-1,-1,-1,-2), \\
\varsigma &= (-1,0,1,0,-1,-1,1,0,1), \\ 
\zeta &= (0,-1,0,1,-1,1,-1,0,1). 
\end{align*}

Let $G=\Aut(M, Q_0)$. Then the generic orbit of the $G$-action on $M$ is open. We set $Q_{0,W} = Q_0\cap W$.  Note that $W$ is $G$-invariant and has $G$-orbits $Q_{0,W}\setminus (l_{01}\cup l_{23})$, $l_{01}=\{X_0=X_1=0\}$, $l_{23}=\{X_2=X_3=0\}$. Let $K=\{X_0=X_1=X_2=X_3=0\} \seq M$; then $K$ is isomorphic to the quadric threefold ${Q^3 = \{Y_4Y_9+Y_8^2+Y_6Y_7=0\}}\seq \IP^4_{Y_4,Y_6,Y_7,Y_8,Y_9}$. Note that $K\seq Q_0$  is $G$-invariant. The hyperplane sections $H^Q_4=\{Y_4=0\}$, $H^Q_9=\{Y_9=0\}\seq K$ are both cones over the conic $q=H^Q_4\cap H^Q_9= \{Y_6Y_7+Y_8^2=0\}\seq \IP^2_{Y_6,Y_7,Y_8}$, with vertices $o_4$, $o_9$, respectively. 
We see that $l_{23} \seq M$ is one of the minimal $G$-orbits. 

\subsection{Properties of quintic Del Pezzo fivefolds}
\label{The properties of Del Pezzo fivefolds}
We make some preparations for computing the delta invariant $\delta(M_5, cQ_0)$ for suitable $c$. Note that $\Aut(M,Q_0)$ is generated by the one-parameter subgroups with weights $-\frac{1}{5}(2\eta+\varsigma+\zeta)$ and $-\frac{1}{5}(\eta-2\varsigma-2\zeta)$ and an $\SL_2$. We consider the toric divisors corresponding to these two one-parameter subgroups, which are $E_W$ and $E_S$, respectively. 

The blowup $\tilde{M}\to M$ naturally embeds in $M\times \IP^4_{Y_4,Y_6,Y_7,Y_8,Y_9}$. Let $\tau\colon \tilde{M}\to \IP^4$ be the natural projection; we have $\tau^*\CO_{\IP^4}(1) \sim \pi^*\CO_M(1)-E_W$. Hence we get the following
bi-homogeneous equations defining $E_W\seq W \times \IP^4_{Y_4,Y_6,Y_7,Y_8,Y_9}$: 
\begin{eqnarray}
\label{Eqnarray. E_W defining equation}
\left\{ \begin{array}{ll}
X_0Y_7+X_1Y_8+X_2Y_4=0, \\
X_0Y_8-X_1Y_6+X_3Y_4=0, \\
X_0Y_9-X_2Y_6-X_3Y_8=0, \\
X_1Y_9-X_2Y_8+X_3Y_7=0.
\end{array} \right.
\end{eqnarray}
If we consider these as equations in $X_0$, $X_1$, $X_2$, $X_3$, then the determinant of the coefficients is $(Y_4Y_9+Y_8^2+Y_6Y_7)^2$. Hence the restriction $\tau_W\colon  E_W \to \IP^4$ is a $\IP^1$-bundle over $Q^3=\{Y_4Y_9+Y_8^2+Y_6Y_7=0\}\seq \IP^4$. One may also consider the blowup of $M$ along $K$ with exceptional divisor $E_K$. Then $E_W$ and $E_K$ are isomorphic over $Q^3$ (as can be seen by comparing the defining equations). 
The inverse image $L_{23}=\pi^{-1}(l_{23})$ is defined by 
\begin{eqnarray*}
\left\{ \begin{array}{ll}
X_0Y_7+X_1Y_8=0, \\
X_0Y_8-X_1Y_6=0 
\end{array} \right. \seq
\,\, l_{23} \times \IP^3_{Y_4,Y_6,Y_7,Y_8}, 
\end{eqnarray*}
which is contained in $\tQ_{0,W}=\pi^{-1}(Q_{0,W})$. 
Hence $L_{23}\cong \IF_2$, and the image of $\tau_{L_{23}}$ is just $H^Q_9\cong \IP(1,1,2)$. The line $l=\{Y_6=Y_7=Y_8=0\}\seq L_{23}$ is contracted by $\tau_{L_{23}}$ to the vertex $o_9$. 
Note that $L_{23}$ is contained in the pull-back $H_9=\tau^{-1}(H_9^Q)\seq W \times \IP^3_{Y_4,Y_6,Y_7,Y_8}$, which is defined by  
\begin{eqnarray*}
\left\{ \begin{array}{ll}
X_0Y_7+X_1Y_8+X_2Y_4=0, \\
X_0Y_8-X_1Y_6+X_3Y_4=0, \\
-X_2Y_6-X_3Y_8=0, \\
-X_2Y_8+X_3Y_7=0.
\end{array} \right.
\end{eqnarray*}

Let $\tE_W\to E_W$ be the blowup of $E_W$ along $l$ with exceptional divisor $E_l\cong l\times \IP^2$. Then the morphism $\tau_W\colon E_W\to Q^3$ lifts to $\tilde{\tau}_W\colon  \tE_W \to \tQ^3$, where $\tQ^3$ is the blowup of $Q^3$ at $o_9$. And the restriction of $\tilde{\tau}_W$ to $E_l$ is just the second projection. Let $\tL_{23},\tilde{\tQ}_{0,W} \seq \tE_W$ be the strict transforms of $L_{23}$ and $\tQ_{0,W}$, respectively. Then the curve $C=\tL_{23}\cap E_l \seq E_l$ is defined as 
\begin{eqnarray*}
\left\{ \begin{array}{ll}
X_0\tY_7+X_1\tY_8=0, \\
X_0\tY_8-X_1\tY_6=0 
\end{array} \right. \seq
\,\, l \times \IP^2_{\tY_6,\tY_7,\tY_8} = E_l, 
\end{eqnarray*}
which is contained in the surface $R=\tilde{\tQ}_{0,W}\cap E_l=\{\tY_6\tY_7+\tY_8^2=0\}\seq E_l$. 
 
\begin{rmk} 
The valuation $\ord_{E_l}$ (over $E_W$) is just the restriction of $\ord_{E_S}$ (over $M$) to $E_W$. 
\end{rmk} 

Finally, let $\tilde{\tE}_W\to \tE_W$ be the blowup of $\tE_W$ along $\tL_{23}$ with exceptional divisor $E_{\tL_{23}}$. Then the strict transform $\tE_l$ of $E_l$ is the blowup of $E_l$ along $C$ with exceptional divisor $E_C=\tE_l\cap E_{\tL_{23}}$. We shall denote by $\tR\seq \tE_l$ the strict transform of $R$ and write $\tC = \tR\cap E_C$. 

Choosing $\lam=0$ and $d=-a$ in Equation~(\ref{Matrix: Aut(M_5,Q_0)}), we see that $G=\Aut(M,Q_0)$ has a subgroup $G_0\cong \SL_2$ generated by 
\begin{eqnarray}
\label{Matrix: Aut(M_5,Q_0)-2}
P=
\left( \begin{array}{ccc}
1 & 0 & 0  \\
0 & A & 0  \\
0 & 0 & (A^T)^{-1}  
\end{array} \right) 
=
\left( \begin{array}{ccccc}
1 & 0 & 0 & 0 & 0 \\
0 & a & b & 0 & 0 \\
0 & c & d & 0 & 0 \\
0 & 0 & 0 & d & -c \\
0 & 0 & 0 & -b & a 
\end{array} \right), 
\end{eqnarray}
for
$A=\begin{psmallmatrix}
a & b  \\
c & d  
\end{psmallmatrix}  \in \SL_2$.
  The $G_0$-action on $M$ lifts to $E_W$, $E_l$ and $\tE_l$ successively. Since the $G_0$-action on $\IP^9_{X_0,\dots,X_9}$ is given by $\wedge^2 P$,  the $G_0$-action on $E_W\seq \IP^3_{X_0,X_1,X_2,X_3} \times \IP^4_{Y_4,Y_6,Y_7,Y_8,Y_9}$ is given by 
\begin{eqnarray}
\label{Eqnarray. matrix E_W}
\left(\left( \begin{array}{cccc}
 a & b & 0 & 0 \\
 c & d & 0 & 0 \\
 0 & 0 & d & -c \\
 0 & 0 & -b & a 
\end{array} \right), 
\left( \begin{array}{ccccc}
1 & 0 & 0 & 0 & 0 \\
0 & a^2 & -b^2 & 2ab & 0 \\
0 & -c^2 & d^2 & -2cd & 0 \\
0 & ac & -bd & ad+bc & 0 \\
0 & 0 & 0 & 0 & 1 
\end{array} \right) \right). 
\end{eqnarray}
Hence the $G_0$-action on $E_l = \IP^1_{X_0,X_1} \times \IP^2_{\tY_6, \tY_7, \tY_8}$ is given by 
\begin{eqnarray}
\label{Eqnarray. matrix E_l}
\left(\left( \begin{array}{cc}
 a & b \\
 c & d \\
\end{array} \right), 
\left( \begin{array}{ccc}
 a^2 & -b^2 & 2ab  \\
 -c^2 & d^2 & -2cd  \\
 ac & -bd & ad+bc  \\
\end{array} \right) \right). 
\end{eqnarray}
We may directly check that the $G_0$-action on $E_l$ has orbit decomposition  
\begin{eqnarray}
\label{Eqnarray. E_l orbit decomposition}
E_l= (E_l\setminus R) \sqcup (R\setminus C) \sqcup C.
\end{eqnarray}
Moreover, the $G_0$-action on $\tE_l = \Bl_C E_l$ has orbit decomposition  
\begin{eqnarray}
\label{Eqnarray. tE_l orbit decomposition}
\tE_l= (\tE_l\setminus (\tR\cup E_C)) \sqcup (\tR\setminus \tC) \sqcup (E_C \setminus \tC) \sqcup \tC. 
\end{eqnarray}

\subsection{Computing $\boldsymbol{S}$-invariants over $\boldsymbol{E_W}$}
We will compute the delta invariant of $(M,cQ_0)$ in the last section for suitable $c$. By $G$-equivariant K-stability, it suffices to compute $\delta_{l_{23}}(M,cQ_0)$.
We take the refinement of $R_\bu=R(M,\CO(1))$ by $E_W$, $E_l$ and $E_C$ successively and get multi-graded linear series $W^{E_W}_{\bu\bu}$, $W^{E_l}_{\bu\bu\bu}$, $W^{E_C}_{\bu\bu\bu\bu}$, respectively. 

We set $L=\pi^*\CO_W(1)$ and $H=\tau^*\CO_{Q^3}(1)$; then $L+\CO_{E_W/W}(1)=H$. Let $\alpha \seq \tL_{23} \seq \tE_W$ be the strict transform of a fiber of $L_{23}\to l_{23}$, and $\beta \seq \tH_9$ be a curve intersecting $l$ transversally at a point such that it is mapped to a ruling of $H^Q_9=\IP(1,1,2)$ isomorphically via $\tau$, and mapped to a line in $W$ isomorphically via~$\pi$. Then $E_l\cdot \alpha=E_l\cdot \beta=1$ and 
\begin{eqnarray*}
\left( \begin{array}{ccccc}
L\cdot\alpha & H\cdot\alpha & L\cdot\beta & H\cdot\beta \\
0 & 1 & 1 & 1 
\end{array} \right). 
\end{eqnarray*}
Recall that $W^{E_W}_{(1,t)}=L+\CO_{E_W/W}(t)=(1-t)L+tH$. Then 
\begin{align*}
\alpha \cdot (W^{E_W}_{(1,t)}-sE_l) 
&= t-s, \\
\beta \cdot (W^{E_W}_{(1,t)}-sE_l) 
&= 1-s. 
\end{align*}
Hence $W^{E_W}_{(1,t)}-sE_l$ is ample for $0<s<t$. We set $\tL=L-E_l$ and $\tH=H-E_l$, where $\tH$ is semiample but $\tL$ has base locus $\tL_{23}$. If $t<s<1$, then $W^{E_W}_{(1,t)}-sE_l$ has asymptotic base locus $\tL_{23}$. Blowing it up, we get $\ttE_W$, and $\ttL=\tL-E_{\tL_{23}}$ is semiample. 
Consequently, we have 
\begin{eqnarray*}
W_{(1,t)}^{E_W} - sE_l = 
\left\{ \begin{array}{ll}
H^0\left(\tE_W, (1-t)L+(t-s)H+s\tH \right),
& 0\le s \le t,\\
H^0\left(\ttE_W, (s-t)\tilde{\tL}+(1-s)L+t\tH \right) + (s-t) E_{\tL_{23}},
& t\le s\le 1, \\
H^0\left(\ttE_W, (1-t)\tilde{\tL}+(1+t-s)\tH \right) + (s-t) E_{\tL_{23}} + (s-1)\ttH_9,
& 1\le s\le 1+t. \\
\end{array} \right.
\end{eqnarray*}
By Equation~(\ref{Formula: Intersection theory on E_W: sigma_1 and h_1}), we have intersection numbers 
\begin{eqnarray}
\label{Formula: Intersection theory on E_W: L and H}
\left( \begin{array}{ccccc}
L^4 & L^3H & L^2H^2 & LH^3 & H^4 \\
0 & 1 & 2 & 2 & 0 
\end{array} \right), \quad 
\left( \begin{array}{cccccc}
L\tH^3 & L^2\tH^2 & L^3\tH & L\tL^3 & L^2\tL^2 & L^3\tL \\
 1 & 2 & 1 & -1 & 0 & 0 
\end{array} \right). 
\end{eqnarray}
Note that $N_{\tL_{23}/\tE_W}\cong \tL^{\oplus2}|_{\tL_{23}}$ and $[\tL_{23}]=[\tL]^2$ in the Chow ring;  we have intersection numbers 
\begin{eqnarray}
\label{Formula: Intersection theory on tE_W: ttL and tH}
\left( \begin{array}{ccccc}
\tilde{\tL}^4 & \tilde{\tL}^3\tH & \tilde{\tL}^2\tH^2 & \ttL\tH^3 & \tH^4 \\
0 & 0 & 0 & 1 & 0 
\end{array} \right), \quad
\left( \begin{array}{cccccc}
\tilde{\tL}^3L & \tilde{\tL}^2L^2 & \ttL L^3 & \tilde{\tL}^2\tH L & \tilde{\tL}\tH^2L & \tilde{\tL} \tH L^2\\
 0 & 0 & 0 & 0 & 1 & 1 
\end{array} \right). 
\end{eqnarray}
Hence 
\begin{eqnarray*}
\vol(W_{(1,t)}^{E_W} - sE_l) = 
\left\{ \begin{array}{ll}
4(1-t)((1-t)^2t+3(1-t)t^2+2t^3-s^3),
& 0\le s \le t,\\
4(s-t)(3(1-s)^2t+3(1-s)t^2+t^3) &\\
+4(1-s)t((1-s)^2+3(1-s)t+t^2),
& t\le s\le 1, \\
4(1-t)(1+t-s)^3,
& 1\le s\le 1+t, 
\end{array} \right. 
\end{eqnarray*}
\begin{eqnarray}
\label{Computation: S(W^E_W;E_l)}
S(W_{\bu\bu}^{E_W}; E_l)
= \frac{1}{5/5!} 
\int_0^1 \int_0^{1+t} \frac{\vol(W_{(1,t)}^{E_W}-sE_l)}{4!} dsdt 
\,=\, \frac{23}{30}. 
\end{eqnarray}

Next we compute $W^{E_l}_{\bu\bu\bu}$ and $S(W^{E_l}_{\bu\bu\bu}; E_C)$. Recall that $E_l \cong l\times \IP^2$. We set $L_E=\pr_1^*\CO_l(1)$ and $H_E=\pr_2^*\CO_{\IP^2}(1)$. Then $L|_{E_l}=L_E$, $H|_{E_l}=0$ and $E_l|_{E_l}=-H_E$. Moreover, $\tH_9|_{E_l}=R$ is a trivial $\IP^1$ bundle over the conic $q\seq\IP^2$, and $\tL_{23}|_{E_l}=C\seq R$. 
The strict transform $\tE_l$ of $E_l$ on $\ttE_W$ is the blowup of $E_l$ along $C$ with exceptional divisor $E_C=E_{\tL_{23}}|_{\tE_l}$. We also have $\ttH_9|_{\tE_l} = \tR$ and $\tR\cap E_C = \tC$. In a word, we have 
\begin{eqnarray}
\label{Formula: restriction of subvarieties of ttE_W to tE_l}
\left( \begin{array}{ccccc}
L & H & E_l & \tH_9 & \tL_{23} \\
L_E & 0 & -H_E & R & C 
\end{array} \right), \quad 
\left( \begin{array}{cc}
E_{\tL_{23}} & \ttH_9 \\
E_C & \tR
\end{array} \right), 
\end{eqnarray}
where the first rows consist of divisors and subvarieties of $\tE_W$ and $\ttE_W$, and the second rows consist of their restrictions to $E_l$ and $\tE_l$, respectively. 
Since $C$ is the complete intersection of two divisors in $|L_E+H_E|$, we set $\tB_E=L_E+H_E-E_C$. Hence
\begin{eqnarray*}
W_{(1,t,s)}^{E_l} = 
\left\{ \begin{array}{ll}
H^0\left(E_l, (1-t)L_E+sH_E \right), 
& 0\le s\le t,\\
H^0\left(\tE_l, (s-t)\tB_E+(1-s)L_E+tH_E \right) + (s-t) E_C,
& t\le s\le 1, \\
H^0\left(\tE_l, (1-t)\tB_E+(1+t-s)H_E \right) + (s-t) E_C + (s-1)\tR,
& 1\le s\le 1+t. \\
\end{array} \right.
\end{eqnarray*}
We may compute $W_{(1,t,s)}^{E_l}-rE_C$: 
\begin{enumerate}
\item If $0\le s\le t, s+t\le 1$, then
\begin{eqnarray*}
W_{(1,t,s)}^{E_l}-rE_C = 
\left\{ \begin{array}{ll}
H^0\left(\tE_l, 
r\tB_E + (1-t-r)L_E + (s-r)H_E \right),
& 0\le r\le s. 
\end{array} \right.
\end{eqnarray*}

\item If $0\le s\le t, s+t\ge1$, then 
\begin{eqnarray*}
W_{(1,t,s)}^{E_l}-rE_C = 
\left\{ \begin{array}{ll}
H^0\left(\tE_l, 
r\tB_E + (1-t-r)L_E + (s-r)H_E \right), 
& 0\le r\le 1-t,\\
H^0\left(\tE_l, 
(1-t)\tB_E + (1-t+s-2r)H_E 
\right) + (-1+t+r) \tR,
& 1-t\le r \le \frac{1}{2}(1-t+s). 
\end{array} \right.
\end{eqnarray*}

\item If $t\le s \le 1, s+t\le 1$, then 
\begin{eqnarray*}
W_{(1,t,s)}^{E_l}-rE_C = 
\left\{ \begin{array}{ll}
H^0\left(\tE_l, 
(s-t)\tB_E + (1-s)L_E+tH_E
\right) + (-t+s-r)E_C,
& 0\le r\le s-t, \\
H^0\left(\tE_l, 
r\tB_E + (1-t-r)L_E + (s-r)H_E \right), 
& s-t \le r \le s. 
\end{array} \right.
\end{eqnarray*}

\item If $t\le s\le 1, s+t\ge1$, then 
\begin{eqnarray*}
W_{(1,t,s)}^{E_l}-rE_C = 
\left\{ \begin{array}{ll}
H^0\left(\tE_l, 
(s-t)\tB_E + (1-s)L_E+tH_E
\right) + (-t+s-r)E_C,
& 0\le r\le s-t, \\
H^0\left(\tE_l, 
r\tB_E + (1-t-r)L_E + (s-r)H_E \right), 
& s-t\le r\le 1-t,\\
H^0\left(\tE_l, 
(1-t)\tB_E + (1-t+s-2r)H_E 
\right) + (-1+t+r) \tR,
& 1-t\le r \le \frac{1}{2}(1-t+s). 
\end{array} \right.
\end{eqnarray*}

\item If $1\le s\le 1+t$, then 
\begin{eqnarray*}
W_{(1,t,s)}^{E_l}-rE_C = 
\left\{ \begin{array}{ll}
H^0\left(\tE_l, 
(1-t)\tB_E + (1+t-s)H_E
\right) & \\
+ (-t+s-r)E_C + (s-1)\tR,
& 0\le r\le s-t, \\
H^0\left(\tE_l, 
(1-t)\tB_E + (1-t+s-2r)H_E 
\right) + (-1+t+r) \tR,
& s-t \le r\le \frac{1}{2}(1-t+s). 
\end{array} \right.
\end{eqnarray*}
\end{enumerate}
In conclusion, we have
\begin{eqnarray*}
W_{(1,t,s)}^{E_l}-rE_C = 
\left\{ \begin{array}{ll}
H^0\left(\tE_l, 
r\tB_E + (1-t-r)L_E + (s-r)H_E \right), 
& (r,s,t) \in \D_A,\\
H^0\left(\tE_l, 
(1-t)\tB_E + (1-t+s-2r)H_E 
\right) + (-1+t+r) \tR,
& (r,s,t)\in \D_B, \\
H^0\left(\tE_l, 
(s-t)\tB_E + (1-s)L_E+tH_E
\right) + (-t+s-r)E_C, 
& (r,s,t)\in \D_{C_1}, \\
H^0\left(\tE_l, 
(1-t)\tB_E + (1+t-s)H_E
\right) & \\
+ (-t+s-r)E_C + (s-1)\tR, 
& (r,s,t)\in \D_{C_2}, 
\end{array} \right.
\end{eqnarray*}
where (always assume $0\le t\le 1$)
\begin{align*}
\D_A 
&= \{0\le s\le t, s+t\le 1, 0\le r\le s\} \\
&\quad  \cup \{0\le s\le t, s+t\ge 1, 0\le r\le 1-t\} \\
&\quad  \cup \{t\le s\le 1, s+t\le 1, s-t\le r\le s\} \\
& \quad  \cup\{t\le s\le 1, s+t\ge 1, s-t\le r\le 1-t\},  \\
\D_B 
&= \{0\le s\le t, s+t\ge 1, 1-t\le r\le \frac{1}{2}(1-t+s)\} \\
&\quad  \cup \{t\le s\le 1, s+t\ge 1, 1-t\le r\le \frac{1}{2}(1-t+s)\} \\
&\quad   \cup\{1\le s\le 1+t, s-t\le r\le \frac{1}{2}(1-t+s)\}, \\
\D_{C_1} 
&= \{t\le s\le 1, 0\le r\le s-t\}, \\
\D_{C_2} 
&= \{1\le s\le 1+t, 0\le r\le s-t\}.  
\end{align*}
Since $N_{C/E_{l}}\cong (L_E+H_E)^{\oplus 2}$, we have intersection numbers
\begin{eqnarray}
\label{Formula: Intersection theory on tE_l: E_C and tLH_E}
\left( \begin{array}{ccccc}
E_C^3 & E_C^2L_E & E_C^2H_E  \\
-6 & -1 & -2  
\end{array} \right), \quad 
\left( \begin{array}{ccccc}
\tB_E L_E^2 & 
\tB_E L_E H_E & 
\tB_E H_E^2 & 
\tB_E^2 \\
 0 & 1 & 1 & 0
\end{array} \right). 
\end{eqnarray}
Hence
\begin{eqnarray*}
\vol(W_{(1,t,s)}^{E_l}-rE_C) = 
\left\{ \begin{array}{ll}
6r(1-t-r)(s-r) + 3(1-t)(s-r)^2,
& (r,s,t) \in \D_A,\\
3(1-t)(1-t+s-2r)^2,
& (r,s,t)\in \D_B, \\
6(s-t)(1-s)t+3(1-t)t^2,
& (r,s,t)\in \D_{C_1}, \\
3(1-t)(1+t-s)^2,
& (r,s,t)\in \D_{C_2},  
\end{array} \right.
\end{eqnarray*}
\begin{eqnarray}
\label{Computation: S(W^E_l;E_C)}
S(W_{\bu\bu\bu}^{E_l}; E_C)
= \frac{1}{5/5!} 
\int_{\D_A\cup\D_B\cup\D_{C_1}\cup\D_{C_2}} 
\frac{\vol(W_{(1,t,s)}^{E_l}-rE_C)}{3!}  drdsdt 
\,=\, \frac{154}{405}. 
\end{eqnarray}

Finally, we compute $W^{E_C}_{\bu\bu\bu\bu}$ and $S(W^{E_C}_{\bu\bu\bu\bu}; \tC)$. 
Recall that $E_C\cong C\times \IP^1$ is a trivial bundle. We set $l_E=\pr_1^*\CO_C(1)$ and $h_E=\pr_2^*\CO_{\IP^1}(1)$. Then $L_E|_{E_C}=l_E$, $H_E|_{E_C}=2l_E$, $E_C|_{E_C} = -h_E +3l_E$ and $\tB_E|_{E_C}=h_E$. We also have $\tR|_{E_C}=\tC$. In a word, we have 
\begin{eqnarray}
\label{Formula: restriction of subvarieties of tE_l to E_C}
\left( \begin{array}{ccccc}
L_E & H_E & E_C & \tB_E & \tR \\
l_E & 2l_E & -h_E+3l_E & h_E & \tC 
\end{array} \right). 
\end{eqnarray}
Then by the above formula of $W_{(1,t,s)}^{E_l}-rE_C$, we see that
\begin{eqnarray*}
W_{(1,t,s,r)}^{E_C} = 
\left\{ \begin{array}{ll}
H^0\left(\tE_C, 
rh_E + (1-t+2s-3r)l_E \right), 
& (r,s,t) \in \D_A,\\
H^0\left(\tE_C, 
(1-t)h_E + 2(1-t+s-2r)l_E 
\right) + (-1+t+r) \tC,
& (r,s,t)\in \D_B. 
\end{array} \right.
\end{eqnarray*}
Moreover, 
\begin{eqnarray*}
W_{(1,t,s,r)}^{E_C}- u\tC = 
\left\{ \begin{array}{ll}
H^0\left(E_C, 
(r-u)h_E + (1-t+2s-3r-u)l_E\right), 
& (u,r,s,t) \in \D^{\tC}_{AA}, \\
H^0\left(\tE_C, 
(1-t)h_E + 2(1-t+s-2r)l_E
\right) & \\ 
+ (-1+t+r-u)\tC, 
& (u,r,s,t)\in \D^{\tC, 0}_{BB}, \\[.2ex]
H^0\left(\tE_C, 
(r-u)h_E + (1-t+2s-3r-u)l_E
\right), 
& (u,r,s,t)\in \D^{\tC}_{BB}, \\
\end{array} \right.
\end{eqnarray*}
where
\begin{align*}
\D^{\tC}_{AA} 
&= 
\{ (u,r,s,t) \in \IR_{\ge 0}\times \D_A \mid u\le r \text{ and } 1-t+2s-3r \}, \\
\D^{\tC,0}_{BB} 
&=
\{ (u,r,s,t) \in \IR_{\ge 0}\times \D_B \mid u\le -1+t+r \}, \\
\D^{\tC}_{BB} 
&=
\{ (u,r,s,t) \in \IR_{\ge 0}\times \D_B \mid -1+t+r\le u\le r \text{ and } 1-t+2s-3r \}. 
\end{align*}
We set $\D^{\tC} = \D^{\tC}_{AA} \cup \D^{\tC}_{BB}$ and $\D^{\tC,0} = \D^{\tC,0}_{BB}$. Hence 
\begin{eqnarray*}
\vol(W_{(1,t,s,r)}^{E_C}- u\tC) = 
\left\{ \begin{array}{ll}
2(r-u)(1-t+2s-3r-u),
& (u,r,s,t) \in \D^{\tC},\\
4(1-t)(1-t+s-2r),
& (u,r,s,t)\in \D^{\tC,0}, 
\end{array} \right.
\end{eqnarray*}
\begin{eqnarray}
\label{Computation: S(W^E_C;tC)}
S(W_{\bu\bu\bu\bu}^{E_C}; \tC)
= \frac{1}{5/5!} 
\int_{\D^{\tC}\cup \D^{\tC,0}} 
\frac{\vol(W_{(1,t,s,r)}^{E_C}-u\tC)}{2!}  dudrdsdt 
\,=\, \frac{469}{4860}. 
\end{eqnarray}

\section{The K-semistable domains of pairs \texorpdfstring{$\boldsymbol{(M,cQ)}$}{(M,cQ)}}
\label{Section: first and last walls}

We consider the wall-crossing problem of the K-moduli spaces of the log Fano pairs $(M, cQ)$ in this section, where $M$ is a quintic Del Pezzo fourfold or fivefold and $Q$ is an ordinary GM threefold or fourfold, respectively. Note that the general theory of K-moduli wall crossing in the proportional case was established in \cite{ADL19, Zho23}.

For $M=M_4$, one may compute delta for the divisorial valuations $\ord_{E_S}$ and $\ord_{Q_3}$. A straightforward computation shows that the smoothness of $Q_3$ implies that $S\nsubseteq Q_3$. So we have
\begin{align*} 
\frac{A_{M_4,cQ_3}(E_S)}{S(-K_{M_4}-cQ_3;E_S)}
&= \frac{25}{9(3-2c)}, \\
\frac{A_{M_4,cQ_3}(Q_3)}{S(-K_{M_4}-cQ_3;Q_3)}
&= \frac{10(1-c)}{3-2c}. \end{align*}
Hence $(M_4,cQ_3)$ is K-semistable only if $\frac{1}{9}\le c \le \frac{7}{8}$. 

For $M=M_5$, we similarly have for $\ord_{E_W}$ and $\ord_{Q_4}$ 
\begin{align*} 
\frac{A_{M_5,cQ_4}(E_W)}{S(-K_{M_5}-cQ_4;E_W)}
&= \frac{15}{8(2-c)}, \\
\frac{A_{M_5,cQ_4}(Q_4)}{S(-K_{M_5}-cQ_4;Q_4)}
&= \frac{6(1-c)}{2-c}. 
\end{align*}
Hence $(M_5,cQ_4)$ is K-semistable only if $\frac{1}{8}\le c \le \frac{4}{5}$. Moreover, we have the following theorem. 

\begin{thm}
\label{Theorem: K-ss domain of DP4 and DP5}
For general $Q_3\in|\CO_{M_4}(2)|$ and $Q_4\in |\CO_{M_5}(2)|$, we have
$$\Kss(M_4,Q_3)=\left[\frac{1}{9}, \frac{7}{8}\right], \quad
\Kss(M_5,Q_4)=\left[\frac{1}{8}, \frac{4}{5}\right]. $$ 
\end{thm}

\subsection{The last wall}
We show that the pair $(M,cQ)$ is K-stable for those $c$ sufficiently close to the last wall by the following theorem, which is a minor strengthening of the cone construction. 

\begin{thm}
\label{Theorem: strengthen of cone}
Let $(X,\D)$ be a log Fano pair of dimension $n$, and let $V\seq X$ be a prime divisor not contained in $\Supp(\D)$ such that $(X,\D+V)$ is lc. We denote by $\D_V$ the different of $\D$ on $V$. Assume that $-K_X-\D\sim_\IQ (1+r) V $ for some rational number $0<r<n$. Then we have the following:
\begin{enumerate}[label={\rm(\alph*)}, ref={\rm\alph*}]
\item\label{T:sc-a} If\, $\delta(V,\D_V) \ge 1$ and $\delta(X,\D) \ge \frac{r(n+1)}{n(r+1)}$, then $(X, \D +(1-\frac{r}{n})V) $ is K-semistable. 
\item\label{T:sc-b}  If\, $\delta(V,\D_V) > 1$ and $\delta(X,\D) > \frac{r(n+1)}{n(r+1)}$, then $(X, \Delta +cV) $ is K-stable for any 
\begin{eqnarray} \label{Formula minor strengthen of cone 1}
\left(1-\frac{r}{n}\right) - \min\{k_{X,\D}, k_{V,\D_V}\}< c < \left(1-\frac{r}{n}\right), 
\end{eqnarray}
where 
\begin{align*}
k_{X,\D} &= (r+1) \left(\delta(X,\D) - \frac{r(n+1)}{n(r+1)}\right) > 0, \\ 
k_{V,\D_V} &= \frac{r}{n}(n+1)\left(\delta(V,\Delta_V)-1\right) > 0. 
\end{align*}
\end{enumerate}
\end{thm}

\begin{rmk} 
\label{Remark. strengthen of cone}
The theorem says that if $\delta(V,\D_V)>1$ and $\delta(X,\D) > \frac{r(n+1)}{n(r+1)}$, then $(X, \Delta + cV) $ is K-stable for any $0\ll c<1-\frac{r}{n}$. 
\end{rmk}

\begin{proof}
For any $c\in \IQ_{\ge 0}$ such that the pair $(X,\D+cV)$ is log Fano, we denote by $V_\bu$ the graded linear series of $-K_X-\D-cV$ and by $W^V_{\bu\bu}$ the refinement of $V_\bu$ by $V$. 

If $x\in V$ is a closed point, then \cite[Theorem 3.2]{AZ22} implies that 
\begin{eqnarray*}
\de_x(X,\D+cV) \ge 
\min\left\{ 
\frac{(1-c)(n+1)}{r+1-c}, 
\de_x(V,\D_V;W^V_{\bu\bu})
\right\}. 
\end{eqnarray*}
Since $V\sim_\IQ \frac{1}{r+1}(-K_X-\D)$, the multi-graded linear series $W^V_{\bu\bu}$ is almost complete in the sense of \cite[Definition 2.16]{AZ22}. It is clear that $W_{\bu\bu}^{V}$ has no fixed part. Hence by \cite[Equation~(3.1)]{AZ22}, we have
\begin{eqnarray*} 
\de_x(V,\D_V;W^V_{\bu\bu}) = 
\de_x(V,\D_V;c_1(W^V_{\bu\bu}))=
\frac{r(n+1)}{n(r+1-c)}\de_x(V,\D_V), 
\end{eqnarray*} 
where 
\begin{align*} 
c_1(W^V_{\bu\bu})
&=\left((-K_X-\D-cV)
-S(-K_X-\D-cV;V)V\right)|_V \\
&=\frac{n(r+1-c)}{r(n+1)}(-K_V-\D_V). 
\end{align*}
Hence 
\begin{eqnarray} \label{Formula minor strengthen of cone 2}
\de_x(X,\D+cV) \ge 
\min\left\{ 
\frac{(1-c)(n+1)}{r+1-c}, 
\frac{r(n+1)}{n(r+1-c)}\de_x(V,\D_V)
\right\}. 
\end{eqnarray}

If $x\notin V$, then for any valuation $v$ whose center $C_X(v)$ contains $x$, we have $C_X(v)\nsubseteq V$. Hence $v(V) = 0$ and  
$A_{X,\D+cV}(v) = A_{X,\D}(v). $
We have 
\begin{align} 
\label{Formula minor strengthen of cone 3}
\de_x(X,\D+cV) 
&= \mathop{\inf}_{C_X(v)\ni x} \frac{A_{X,\D+cV}(v)}{S(-K_X-\D-cV;v)} \\
\nonumber
&= \mathop{\inf}_{C_X(v)\ni x} \frac{A_{X,\D}(v)}{\frac{r+1-c}{r+1}S(-K_X-\D;v)} 
=
\frac{r+1}{r+1-c}\delta_x(X,\D). 
\end{align}

Note that 
\begin{align*} 
\frac{(1-c)(n+1)}{r+1-c}> 1 
&\Longleftrightarrow c< \left(1-\frac{r}{n}\right), \\
\frac{r(n+1)}{n(r+1-c)}\de(V,\D_V) > 1 
&\Longleftrightarrow c> \left(1-\frac{r}{n}\right) - k_{V,\D_V}, \\
\frac{r+1}{r+1-c}\delta(X,\D) > 1
&\Longleftrightarrow c> \left(1-\frac{r}{n}\right) - k_{X,\D}. 
\end{align*}
We get~\eqref{T:sc-b} by (\ref{Formula minor strengthen of cone 2}) and (\ref{Formula minor strengthen of cone 3}). Replacing the strict inequalities by the non-strict ones, we get~\eqref{T:sc-a}. 
\end{proof}

\begin{rmk}
\label{Remark: Proof of last walls}
Apply Theorem~\ref{Theorem: strengthen of cone} to $(M_4,cQ)$. We have shown that $\delta(M_4)=\frac{25}{27}$. Hence $(r+1)\big(1-\delta(M_4)\big) = \frac{1}{9} < \frac{7}{8}$. By \cite[Theorem 5.12]{AZ21}, we know that the smooth GM threefold $Q$ is K-stable. We conclude that $(M_4,cQ)$ is K-stable for any $0 \ll c<\frac{7}{8}$. 

With the same argument in the proof of Theorem~\ref{Theorem: general special GM manifolds are K-semistable}, we get a K-stable ordinary GM fourfold $Q\seq M_5$. Apply Theorem~\ref{Theorem: strengthen of cone} to $(M_5, cQ)$. Since $\delta(M_5)=\frac{15}{16}$, we have $(r+1)\big(1-\delta(M_5)\big) = \frac{1}{8} < \frac{4}{5}$. Hence $(M_5,cQ)$ is K-stable for any $0 \ll c<\frac{4}{5}$. 
\end{rmk}

\subsection{The first wall}
Assume $M=M_4$. If $c<\frac{1}{9}$, then $(M,cQ)$ is destabilized by $E_S$. We consider the pair $(M,\frac{1}{9}Q)$ with $Q\cap S = C_2$. The product test configuration induced by $\xi_1=-\frac{1}{5}(\eta+\zeta)$ (see Section~\ref{Subsection: Aut of M_4}) degenerates $Q$ to the cone $Q_0=\{X_1^2-4X_0X_4=0\}\seq M$. 

\begin{thm}
\label{Theorem: (M_4,(1/9)Q_0) K-polystable}
The pair $(M_4, \frac{1}{9}Q_0)$ is K-polystable. 
\end{thm}

\begin{proof}
Let $G:=\Aut(M,Q_0)$. It has one-dimensional center $\IT$ whose action is generated by $\xi_1=-\frac{1}{5}(\eta+\zeta)$. Let $N\cong \IZ$ be the lattice generated by $\xi_1$. Recall that the divisor $E_S\seq \tM$ is $G$-invariant. It is the toric divisor of $-\frac{1}{5}(\eta+\zeta)$, and the $\IT$-action on $E_S$ is trivial. We naturally have a $G$-equivariant surjective rational map $f\colon  \tilde{M}\dashrightarrow E_S$, which makes $E_S$ a quotient of $\tilde{M}$ by the $\IT$-action. Hence $\Val^{\IT,\circ}_M \cong \Val^{\circ}_{E_S} \times N_\IR$. Moreover, $\Val^{G,\circ}_M \cong \Val^{(G/\IT),\circ}_{E_S} \times N_\IR$ by Lemma~\ref{Lemma. twist of G-invariant valuations}. 
Via this isomorphism, for any valuation $v\in \Val^{\circ}_{E_S} \seq \Val^{\IT,\circ}_{M}$, we have 
\begin{eqnarray}
\label{Eqnarray. A_M(v) = A_ES(v)}
A_{M, \frac{1}{9}Q_0}(v)=A_{E_S, \frac{1}{9}\tilde{T}}(v). 
\end{eqnarray}
Recall that $\tau_{E_S}\colon E_S\to \IP^3$ is the blowup of a twisted cubic curve $C_3$ with exceptional divisor $F$, and $\pi_{E_S}\colon E_S \to \IP^2$ is a $\IP^1$-bundle with $\tilde{T}=\pi_{E_S}^{-1}(C_2)$, which is the strict transform of the tangent developable $T\seq \IP^3$ of $C_3$. The divisors $F$ and $\tilde{T}$ in $E_S$ are tangent along a curve $C$ of order $2$. They are the only $G$-invariant subvarieties of $E_S$. Also note that $\pi_*^{-1}(Q_0)\cap E_S = \tilde{T}$, intersecting transversally. 

We take the refinements of $R_\bu=R(M,\CO(1))$ by $E_S$ and $F$ successively and get multi-graded linear series $W^{E_S}_{\bu\bu}$ and $W^{F}_{\bu\bu\bu}$, respectively. We may rescale the line bundle $-K_M-\frac{1}{9}Q_0\sim_\IQ\frac{25}{9}\CO(1)$ and set $\tR_\bu=R(M,-K_M-\frac{1}{9}Q_0)$. We denote the corresponding refinements of $\tR_\bu$ by $\tW^{E_S}_{\bu\bu}$ and $\tW^{F}_{\bu\bu\bu}$. By Equations~(\ref{Computation: S(W^E_S;F)}) and (\ref{Computation: S(W^F; C)}), we have
\begin{align*}
S(\tR_\bu;E_S)
&= \frac{25}{9} S(R_\bu;E_S) 
= 2, \\
S(\tW_{\bu\bu}^{E_S}; F) 
&= \frac{25}{9} S(W_{\bu\bu}^{E_S}; F)
=\frac{5}{9}, \\
S(\tW_{\bu\bu\bu}^{F}; C)
&=\frac{25}{9}S(W_{\bu\bu\bu}^{F}; C)
=\frac{5}{9}. 
\end{align*}
Hence $\beta_{M,\frac{1}{9}Q_0}(E_S) = A_{M,\frac{1}{9}Q_0}(E_S)- S(\tR_\bu;E_S)=0$. In other words, $\Fut(\xi_1)=0$. 
Since the restriction of $\frac{1}{9}\tilde{T}$ on $F$ is $\frac{2}{9}C$, we have  
\begin{eqnarray*}
\frac{A_{E_S, \frac{1}{9} \tilde{T}}(F)}
{S(\tW_{\bu\bu}^{E_S}; F)}
=\frac{9}{5}, \quad
\frac{A_{F,\frac{2}{9}C}(C)}
{S(\tW_{\bu\bu\bu}^{F}; C)}
=\frac{7}{5}. 
\end{eqnarray*}
By Theorem~\ref{Theorem: Abban-Zhuang estimate}, we get the following estimate: 
\begin{eqnarray}
\label{Eqnarray. delta_C(E_S,..) > 1}
\delta_{C}(E_S,\frac{1}{9}\tilde{T}; \tW^{E_S}_{\bu\bu})
\,\ge\, 
\min\left\{ 
\frac{A_{E_S,\frac{1}{9}\tilde{T}}(F)}{S(\tW^{E_S}_{\bu\bu};F)}, 
\frac{A_{F,\frac{2}{9}C}(C)}{S(\tW^F_{\bu\bu\bu};C)}
\right\} =
\min\left\{
\frac{9}{5}, \frac{7}{5}
\right\} \,>\, 1. 
\end{eqnarray}
Note that $R_\bu=R(M,\CO(1))$ admits a weight decomposition $R_m=\oplus_{\alpha \in \IZ}R_{m,\alpha}$ with respect to the $\IT$-action. Since $E_S$ is a toric divisor of $\IT$, we have a canonical $\IT$-equivariant isomorphism of $\IN^2$-graded linear series $R_{\bu\bu}\cong W^{E_S}_{\bu\bu}$. For any $v\in\Val^{\IT}_{E_S}$, the filtrations induced by $v$ on $R_{\bu\bu}$ and $W^{E_S}_{\bu\bu}$ are the same via this isomorphism. Hence
\begin{eqnarray}
\label{Eqnarray. S(R; v)=S(W^E_S; v)}
{S(\CO(1);v)}=S(W^{E_S}_{\bu\bu};v). 
\end{eqnarray}

Now we show that $(M,\frac{1}{9}Q_0)$ is K-polystable. Otherwise, then by \cite[Corollary 4.11]{Zhu21}, it is not $G$\nobreakdash-equivariantly K-polystable. Hence by \cite[Theorem 7]{LX14}, there exists a $G$-equivariant non-product-type special test configuration $(\CX,\D,\CL)$ of $(M, \frac{1}{9}Q_0)$ with $\Fut(\CX,\D,\CL) \le 0$. Then $w= \ord_{\CX_0}|_{\CX_1}$ is a $G$-invariant valuation on $M$ not of the form $\wt_\xi$ for any $\xi\in N_\IR$ with $\beta_{M, \frac{1}{9}Q_0}(w) \le 0$. Since $\Val^{G,\circ}_M \cong \Val^{(G/\IT),\circ}_{E_S} \times N_\IR$, we may write $w=v_{a\xi}$ for some $v\in \Val^{(G/\IT),\circ}_{E_S}$ and $a\in\IR$. Let $\theta=\theta_{a\xi_1}(v)$ as defined by Equation~(\ref{Eqnarray: A(v_xi) = A(v)+theta_xi(v)}). We have 
$$1
\ge\frac{A_{M, \frac{1}{9}Q_0}(v_{a\xi})}{S(\tR_\bu;v_{a\xi})}
=\frac{A_{M, \frac{1}{9}Q_0}(v)+\theta}{S(\tR_\bu;v)+\theta}
=\frac{A_{E_S, \frac{1}{9}\tT}(v)+\theta}{S(\tW^{E_S}_{\bu\bu
};v)+\theta}, 
$$
where the first equality follows from \cite[Proposition 3.12]{Li19} and the fact that $\Fut(\xi_1)=0$, and the second equality follows from Equations~(\ref{Eqnarray. A_M(v) = A_ES(v)}) and (\ref{Eqnarray. S(R; v)=S(W^E_S; v)}). 
By the elementary equivalence  
$\frac{a}{b} \le \frac{a'}{b'} \Leftrightarrow \frac{a}{b} \le \frac{a+a'}{b+b'}  \le \frac{a'}{b'}$ (for $b,b'>0$), we have
$$1
\ge\frac{A_{E_S, \frac{1}{9}\tT}(v)}{S(\tW^{E_S}_{\bu\bu
};v)} \ge \delta_{C}(E_S,\frac{1}{9}\tilde{T}; \tW^{E_S}_{\bu\bu}) >1, 
$$
where the second inequality follows from the $(G/\IT)$-invariance of $v$ and $C\seq E_S$ being the unique minimal $(G/\IT)$-orbit (see Equation~(\ref{Eqnarray. E_S orbit decomposition})). We get a contradiction. Hence $(M,\frac{1}{9}Q_0)$ is K-polystable. 
\end{proof}

Assume $M=M_5$. If $c<\frac{1}{8}$, then $(M,cQ)$ is destablized by $E_W$. Let $Q_0=\{X_0X_2+X_1X_3=0\}\seq M$. We consider those $Q$ with $Q_W = Q_{0,W}$. Then the product test configuration of $M$ induced by $-\frac{1}{5}(2\eta+\varsigma+\zeta)$ (see Section~\ref{Subsection: Aut of M_5}) degenerates $Q$ to the cone $Q_0$. We have the following. 

\begin{thm}
\label{Theorem: (M_5,(1/8)Q_0) K-polystable}
The pair $(M_5, \frac{1}{8}Q_0)$ is K-polystable. 
\end{thm}

\begin{proof}
Let $G=\Aut(M,Q_0)$, and let $\IT_1,\IT_2 \seq G$ be the one-parameter subgroups generated by $\xi_1=-\frac{1}{5}(2\eta+\varsigma+\zeta)$ and $\xi_2=-\frac{1}{5}(\eta-2\varsigma-2\zeta)$, respectively, and $N$ the co-weight lattice of the action of $\IT=(\IT_1\times \IT_2)$. By Lemma~\ref{Lemma. center of aut(M_5,Q_0)}, $\IT$ lies in the center of $G$. Note that $E_W$ is a toric divisor of $\IT_1$, and we  naturally have a $\IT_1$-equivariant surjective rational map $f_1\colon  \tM\dashrightarrow E_W$, which makes $E_W$ a quotient of $\tilde{M}$ by the $\IT_1$-action. The one-parameter subgroup $\IT_2$ acts on $E_W$ faithfully, and $E_l$ is a toric divisor of such an action. Hence we have a $\IT_2$-equivariant surjective rational map $f_2\colon  \tE_W \dashrightarrow E_l$ which makes $E_l$ a quotient of $\tE_W$ by the $\IT_2$-action.
Hence $\Val^{\IT,\circ}_M \cong \Val^{\circ}_{E_l} \times N_\IR$. Moreover, we have $\Val^{G,\circ}_M \cong \Val^{(G/\IT),\circ}_{E_l} \times N_\IR$ by Lemma~\ref{Lemma. twist of G-invariant valuations}. Via this isomorphism, for any $v\in \Val^{\circ}_{E_l}\seq \Val^{\IT,\circ}_M$, we have 
\begin{eqnarray}
\label{Eqnarray. A_M(v)=A_El(v)}
A_{M,\frac{1}{8}Q_0}(v) = A_{E_W, \frac{1}{8}Q_{0,W}}(v) = A_{E_l, \frac{1}{8}R}(v) = A_{\tE_l,\frac{1}{8}\tR-\frac{7}{8}E_C}(v). 
\end{eqnarray}
Recall that $E_l \cong \IP^1\times \IP^2$. The $(G/\IT)$-action on $E_l$ has orbit decomposition $E_l= (E_l\setminus R) \sqcup (R\setminus C) \sqcup C$ (see Equation~(\ref{Eqnarray. E_l orbit decomposition})). This $(G/\IT)$-action lifts to $\tE_l=\Bl_C E_l$. Recall that $E_C$ is the exceptional divisor of the blowup, $\tR$ is the strict transform of $R\seq E_l$, and they intersection transversally along $\tC$. They are the only $(G/\IT)$-invariant subvarieties of $\tE_l$ (see Equation~(\ref{Eqnarray. tE_l orbit decomposition})). 

We take the refinements of $R_\bu=R(M,\CO(1))$ by $E_W$, $E_l$ and $E_C$ successively (see Section~\ref{The properties of Del Pezzo fivefolds}) and get multi-graded linear series $W^{E_W}_{\bu\bu}$, $W^{E_l}_{\bu\bu\bu}$, $W^{E_C}_{\bu\bu\bu\bu}$, respectively. Note that $-K_M-\frac{1}{8}Q_0\sim_\IQ \frac{15}{4}\CO(1)$. We also set $\tR_\bu = R(M,-K_M-\frac{1}{8}Q_0)$ and denote by $\tW^{E_W}_{\bu\bu}$, $\tW^{E_l}_{\bu\bu\bu}$, $\tW^{E_C}_{\bu\bu\bu\bu}$ the corresponding refinements. By Equations~(\ref{Computation: S(O(1);E_W)}) and (\ref{Computation: S(W^E_W;E_l)}), we have 
\begin{align*}
A_{M,\frac{1}{8}Q_0}(E_W) - S(\tR_\bu;E_W) 
&= (2-0) - \frac{15}{4} \cdot \frac{8}{15} = 0, \\
A_{E_W,\frac{1}{8}\tQ_{0,W}}(E_l) - S(\tW^{E_W}_{\bu\bu};E_l) 
&= (3-\frac{1}{8}) - \frac{15}{4} \cdot \frac{23}{30} = 0. 
\end{align*}
In other words, $\Fut(\xi_1)=\Fut(\xi_2)=0$. 
Note that $(\tE_l,\frac{1}{8}\tR-\frac{7}{8}E_C)$ is sub-klt and $(\tE_l,\frac{1}{8}\tR+E_C)$ is log smooth. 
By Theorem~\ref{Theorem: Abban-Zhuang estimate} and Equations~(\ref{Computation: S(W^E_l;E_C)}) and (\ref{Computation: S(W^E_C;tC)}), we have 
\begin{align} 
\delta_{\tC}(\tE_l,\frac{1}{8}\tR-\frac{7}{8}E_C; \tW^{E_l}_{\bu\bu\bu})
&\ge 
\min\left\{
\frac{A_{\tE_l,\frac{1}{8}\tR-\frac{7}{8}E_C}(E_C)}{S(\tW^{E_l}_{\bu\bu\bu};E_C)}, 
\frac{A_{E_C,\frac{1}{8}\tC}(\tC)}{S(\tW^{E_C}_{\bu\bu\bu\bu};\tC)}
\right\}\\
\label{Eqnarray. delta(E_l, (1/8)R; W) > 1}
&=
\min\left\{
\frac{2-\frac{1}{8}}{\frac{15}{4}\cdot \frac{154}{405}}, 
\frac{1-\frac{1}{8}}{\frac{15}{4}\cdot \frac{469}{4860}}
\right\} 
\,=\,
\min\left\{
\frac{405}{308}, 
\frac{162}{67}
\right\} 
> 1. 
\end{align}
Since $E_W\seq \tM$ and $E_l\seq \tE_W$ are toric divisors, the weight decomposition gives us a canonical $\IT$-equivariant isomorphism of $\IN^3$-graded linear series $R_{\bu\bu\bu}\cong W^{E_l}_{\bu\bu\bu}$. The filtrations induced by $v\in \Val_{E_l}\seq \Val_M$ on $R_{\bu\bu\bu}$ and $W^{E_l}_{\bu\bu\bu}$ are the same via this isomorphism. Hence 
\begin{eqnarray}
\label{Eqnarray. S(R; v)=S(W^E_l; v)}
{S(\CO(1);v)}=S(W^{E_l}_{\bu\bu\bu};v). 
\end{eqnarray}

Now we show that $(M, \frac{1}{8}Q_0)$ is K-polystable. Otherwise, by the same argument as for Theorem~\ref{Theorem: (M_4,(1/9)Q_0) K-polystable}, there exist a $G$-invariant divisorial valuation $v \in \Val^{\circ}_{E_l}$ and a $\xi\in N_\IQ$ such that $w=v_\xi\in\Val^{G,\circ}_{M}$ is a special divisorial valuation and $\beta_{M, \frac{1}{8}Q_0}(w)\le 0$. Then
$$1
\ge\frac{A_{M, \frac{1}{8}Q_0}(v_{\xi})}{S(\tR_\bu;v_{\xi})}
=\frac{A_{M, \frac{1}{8}Q_0}(v)+\theta}{S(\tR_\bu;v)+\theta}
=\frac{A_{\tE_l,\frac{1}{8}\tR-\frac{7}{8}E_C}(v)+\theta}{S(\tW^{E_l}_{\bu\bu\bu};v)+\theta}. 
$$
Also by the same argument as for Theorem~\ref{Theorem: (M_4,(1/9)Q_0) K-polystable}, we have
$$1
\ge \frac{A_{\tE_l,\frac{1}{8}\tR-\frac{7}{8}E_C}(v)}{S(\tW^{E_l}_{\bu\bu\bu};v)}
\ge \delta_{\tC}(\tE_l,\frac{1}{8}\tR-\frac{7}{8}E_C; \tW^{E_l}_{\bu\bu\bu}) >1, 
$$
which gives a contradiction. Hence $(M, \frac{1}{8}Q_0)$ is K-polystable. 
\end{proof}

\begin{proof}[Proof of Theorem~\ref{Theorem: K-ss domain of DP4 and DP5}]
This follows directly from Remark~\ref{Remark: Proof of last walls} and Theorems~\ref{Theorem: (M_4,(1/9)Q_0) K-polystable} and~\ref{Theorem: (M_5,(1/8)Q_0) K-polystable}. 
\end{proof}

\begin{rmk}
The wall crossings for the K-moduli spaces of $(M_4, cQ_3)$ and $(M_5, cQ_4)$ are closely related to K-moduli spaces of Gushel--Mukai threefolds and fourfolds. In fact, the K-moduli space of GM threefolds (resp.\ GM fourfolds) is isomorphic to the K-moduli space of $(M_4, \frac{7}{8}Q_3)$ (resp.\ $(M_5, \frac{4}{5}Q_4)$) using degeneration to the normal cone. Moreover, if we restrict to $c = \frac{1}{2}$, then the corresponding K-moduli spaces of pairs are isomorphic to K-moduli spaces of special GM fourfolds and fivefolds.
\end{rmk}

\section{K\"ahler--Ricci solitons on quintic Del Pezzo fourfolds and fivefolds}
\label{Section: solitons on DP4 and DP5}

In this section, we show that $M_4$ and $M_5$ both admit K\"ahler--Ricci solitons following the argument in \cite{MW23}, which answers a question asked by \cite{Ino19,Del22}. 
By the works \cite{HL23,BLXZ23}, we know that a Fano manifold $X$ admits a K\"ahler--Ricci soliton if and only if $(X,\xi_0)$ is weighted K-polystable for some holomorphic vector field $\xi_0$ on $X$ (which is called a {\it soliton candidate}l see for example \cite[Section 3.1]{MW23}). 

For Del Pezzo fourfolds $M_4$, consider the $\IG_m$-action induced by $-\frac{1}{5}(\eta+\zeta)$ (see Section~\ref{Subsection: Aut of M_4}) with toric divisor $E_S$. Since $A_{M_4}(E_S)=2$ and $|\CO(1)-u E_S|_\IQ\ne \varnothing$ for $u\in[0,2]$, the moment polytope of this $\IG_m$-action on $-K_{M_4}=\CO(3)$ is $\BP = [-2,4]\ni 3u-2 =: u'$ (the moment polytope is determined by Equation~(\ref{Eqnarray: valuation of product TC})). Recall that the soliton candidate $\xi_0$ with respect to this $\IG_m$-action is the minimizer of the $\BH$-functional 
$$\BH(\xi)=\log\left(\int_\BP e^{-u'\xi } \cdot \DH_\BP(du')\right),$$ 
where $\xi\in \IN(\IG_m)_\IR\cong \IR$ and $\DH_\BP(du') = 3^4\cdot \vol(W^{E_S}_{(1,u)})du$. We can numerically solve 
$$\xi_0 \approx 0.16838665311714196.$$ 

\begin{thm}\label{Theorem: M_4 soliton}
The pair $(M_4,\xi_0)$ is weighted K-polystable. 
\end{thm}

\begin{proof}
By the definition of $\xi_0$, we have 
$$\int_0^2 (3u-2)e^{-(3u-2)\xi_0} \cdot \vol(W^{E_S}_{(1,u)}) du = 0. $$
Recall that the $g$-weighted volume is defined by 
\begin{align*}
\frac{\vol^g(\CF_{E_S}^{(u_1)}R_\bu)}{4!} 
&:= \int_{u\ge u_1} e^{-(3u-2)\xi_0} \frac{\vol(W^{E_S}_{(1,u)})}{3!}du, \\
\Bv^g 
&:= \frac{\vol^g(R_\bu) }{4!}
\,=\, 
\int_0^2 e^{-(3u-2)\xi_0} \frac{\vol(W_{(1,u)})}{3!}du \\
&\approx 0.2055698662861948. 
\end{align*}
Let $R_\bu = R(M_4, \CO(1))$. Using integration by parts, we have 
\begin{eqnarray*}
S^g(R_\bu;E_S) 
:= \frac{1}{\Bv^g} 
\int_{0}^2 
\frac{\vol^g(\CF_{E_S}^{(u)}R_\bu)}{4!} du 
= \frac{1}{\Bv^g} 
\int_0^2 ue^{-(3u-2)\xi_0}
\frac{\vol(W^{E_S}_{(1,u)})}{3!}du 
= \frac{2}{3} 
= \frac{1}{3} A_{M_4}(E_S). 
\end{eqnarray*}
Moreover, we have (see Equations~(\ref{Region. M_4 W^ES F (u,v)}) for $\D_1^F$ and $\D_2^F$, and Equation~(\ref{Region. M_4 W^F C}) for $\D^C$)
\begin{align*}
S^g(W_{\bu\bu}^{E_S}; F)
&= \frac{1}{\Bv^g} 
\int_{\D_1^F\cup \D_2^F} 
e^{-(3u-2)\xi_0} \cdot
\frac{\vol\left(\CF_F^{(v)}W_{(1,u)}^{E_S}\right)}{3!} dudv 
\,\approx\, 0.179638, \\
S^g(W_{\bu\bu\bu}^{F}; C)
&= \frac{1}{\Bv^g}
\int_{\D^C} 
e^{-(3u-2)\xi_0} \cdot
\frac{\vol\left(\CF_C^{(w)}W_{(1,u,v)}^{F}\right)}{2!} dudvdw 
\,\approx\, 0.211933. 
\end{align*}
By \cite[Corollary 5.5]{MW23}, we have
\begin{align*}
\delta^g_{C}(E_S;\tW^{E_S}_{\bu\bu})
&\ge \frac{1}{3}
\min\left\{
\frac{A_{E_S}(F)}{S^g(W^{E_S}_{\bu\bu};F)}, 
\frac{A_{F}(C)}{S^g(W^F_{\bu\bu\bu};C)}
\right\} \\
&= \frac{1}{3}
\min\left\{
\frac{1}{0.179638}, \frac{1}{0.211933}
\right\} \,>\,1. 
\end{align*}
Then following the same argument as for Theorem~\ref{Theorem: (M_4,(1/9)Q_0) K-polystable} (see also \cite{MW23,MW24}), we see that $(M_4,\xi_0)$ is weighted K-polystable. 
\end{proof}

For Del Pezzo fivefolds $M_5$, consider the $\IG_m$-action induced by $-\frac{1}{5}(2\eta+\varsigma+\zeta)$ (see Section~\ref{Subsection: Aut of M_5}) with toric divisor $E_W$. Since $A_{M_5}(E_W)=2$ and $|\CO(1)-t E_W|_\IQ\ne \varnothing$ for $t\in[0,1]$, the moment polytope of this $\IG_m$-action on $-K_{M_5}=\CO(4)$ is $\BP = [-2,2]\ni 4t-2 =: t'$. We also numerically solve the soliton candidate 
$$\eta_0 \approx 0.1693945440748772.$$ 

\begin{thm}\label{Theorem: M_5 soliton}
The pair $(M_5,\eta_0)$ is weighted K-polystable. 
\end{thm}

\begin{proof}
By the same argument as in the proof of Theorem~\ref{Theorem: M_4 soliton}, we have 
\begin{align*}
\Bv^g 
&= 
\int_0^1 e^{-(4t-2)\eta_0} 
\frac{\vol(W^{E_W}_{(1,t)})}{4!}dt \\
&\approx 0.04119805228615477, \\
S^g(R_\bu; E_W)
&= \frac{1}{4}A_{M_5}(E_W). 
\end{align*}
Moreover, we have 
\begin{align*}
S^g(W_{\bu\bu}^{E_W}; E_l)
&= \frac{1}{\Bv^g} 
\int_0^1 
e^{-(4t-2)\eta_0} 
\int_0^{1+t} 
\frac{\vol(W_{(1,t)}^{E_W}-sE_l)}{4!} dsdt 
\,=\, \frac{3}{4}, \\
S^g(W_{\bu\bu\bu}^{E_l}; E_C) 
&= \frac{1}{\Bv^g} 
\int_{\D_A\cup\D_B\cup\D_{C_1}\cup\D_{C_2}} 
e^{-(4t-2)\eta_0} \cdot
\frac{\vol(W_{(1,t,s)}^{E_l}-rE_C)}{3!}  drdsdt 
\,\approx\, 0.390484, \\
S^g(W_{\bu\bu\bu\bu}^{E_C}; \tC)
&= \frac{1}{\Bv^g} 
\int_{\D^{\tC}\cup\D^{\tC,0}} 
e^{-(4t-2)\eta_0} \cdot
\frac{\vol(W_{(1,t,s,r)}^{E_C}-u\tC)}{2!}  dudrdsdt 
\,\approx\, 0.089469. 
\end{align*}
Then 
\begin{align*} 
\delta^g_{\tC}(\tE_l,-E_C; \tW^{E_l}_{\bu\bu\bu})
&\ge 
\frac{1}{4}\min\left\{
\frac{A_{\tE_l, -E_C}(E_C)}{S^g(W^{E_l}_{\bu\bu\bu};E_C)}, 
\frac{A_{E_C}(\tC)}{S^g(W^{E_C}_{\bu\bu\bu\bu};\tC)}
\right\}\\
&=
\frac{1}{4}\min\left\{
\frac{2}{0.390484}, \frac{1}{0.089469}
\right\} \,>\, 1. 
\end{align*}
Now the same argument as for Theorems~\ref{Theorem: (M_5,(1/8)Q_0) K-polystable} and~\ref{Theorem: M_4 soliton} shows that $(M_5,\eta_0)$ is weighted K-polystable. 
\end{proof}

\section{K-stability of general special Gushel--Mukai manifolds}
\label{Section: K-stability of general special GM}

In this section, we prove the main theorem of this paper, that is, general special GM manifolds are K-stable. 
Let $X$ be a GM $n$-fold and $M\seq \IP^{n+3}$ be the quintic Del Pezzo $n$-fold. When $X$ is special, there is a double cover $p\colon  X \to M$ branched along a smooth quadric divisor $Q$ which is an ordinary GM $({n-1})$-fold. We have $p^*(K_{M}+\frac{1}{2}Q) = K_{X}$. By \cite[Theorem 1.2]{LZ20}, we need to test the K-stability of $(M,\frac{1}{2}Q)$. 

We first recall some basic results about the projective cone over a log Fano pair. Let $(V, \D_V)$ be an $(n-1)$-dimensional log Fano pair such that $L=-\frac{1}{r}(K_V+\D_V)$ is an ample Cartier divisor for some rational number $r\in (0, n]$. 
We define the projective cone over $V$ with polarization $L$ by 
$$Y=\overline{\CC}(V,L):=\Proj\left(\bigoplus_{m\ge0}\bigoplus_{0\le\lam\le m}H^0(V, (m-\lam)L)s^\lam\right), $$
which is the union of the affine cone $\CC=\Spec\big(\oplus_{m\ge0}H^0(V,  m L)\big)$ and a divisor $V_\infty \cong V$ at infinity. Let $\D_Y$ be closure of $\D_V\times \IC^*$ in $Y$. Then $-(K_Y+\D_Y)$ is Cartier and 
$$-(K_Y+\D_Y)\sim_\IQ(1+r)V_\infty. $$
By \cite[Proposition 5.3]{LX20} or Theorem~\ref{Theorem: strengthen of cone}, we have the following. 

\begin{prop}\label{K-semistability of cone}
If\, $(V, \D_V)$ is K-semistable, then $(Y, \D_Y+(1-\frac{r}{n})V_\infty)$ is K-semistable. 
\end{prop}

We will use the following result about degeneration to the normal cone. See for example \cite[Lemma~2.12]{LZ22} and \cite[Lemma 5.3]{ZZ22}. 

\begin{prop} \label{Proposition: degeneration to normal cone}
  Let $X$ be a Fano variety of dimension $n$, and $V$ be a prime divisor on $X$ such that $-K_X\sim_\IQ (1+r)V$ with $0 < r\le n$. Then $(X, V)$ specially degenerates to $(Y, V_\infty)$, where $Y=\overline{\CC}(V, H)$ and $H=V|_V$. 
In particular, if\, $V$ is K-semistable, then $(X, (1-\frac{r}{n})V)$ is K-semistable.  
\end{prop}

Now we prove the main theorem of this paper. 

\begin{thm} \label{Theorem: general special GM manifolds are K-semistable}
General special GM-manifolds are K-stable. 
\end{thm}

\begin{proof}
It was proved by \cite{AZ22} that smooth GM threefolds are all K-stable. We need to prove that there exists a K-stable special GM $n$-fold for $n=4,5,6$. Then by the openness of K-stability, we are done.  
We first prove that general special GM $n$-folds are K-semistable for $n=4,5,6$. It suffices to find a smooth quadric divisor $Q_{n-1}\seq M_n$ such that $(M_n,\frac{1}{2}Q_{n-1})$ is K-semistable. 

For $n=4$, note that $Q$ is a GM threefold, which is K-stable by \cite[Theorem C]{AZ21}. Since $-K_M \sim_\IQ \frac{3}{2}Q$, we have $r=\frac{1}{2}$ and $1-\frac{r}{n}=\frac{7}{8}=:c_\max$. By Proposition~\ref{Proposition: degeneration to normal cone}, we conclude that $(M, \frac{7}{8}Q)$ is K-semistable. On the other hand, we consider a smooth hyperplane section $H$ of $M$. Then $H\cong M_3$, which is K-polystable by \cite[Example 3.4.1]{ACC+}. BY the same method, since $-K_M=3H$, we have $r=2$, $1-\frac{r}{n}=\frac{1}{2}$. Hence $(M, \frac{1}{2}H)$ is K-semistable. One can construct a family that deforms $(M, \frac{1}{2}H)$ to $(M, \frac{1}{4}Q)$ for some smooth quadric divisor $Q$ of $M$. By the openness of K-semistability, we conclude that there exists a smooth quadric divisor $\tQ$ on $M$ such that $(M, \frac{1}{4}\tQ)$ is K-semistable. Finally, by interpolation of K-stability (see \cite[Proposition~2.13]{ADL19}), we see that $(M, c\tQ)$ is K-semistable for any $\frac{1}{4} \le c\le \frac{7}{8}$. In particular, $(M, \frac{1}{2}\tQ)$ is K-semistable. 

For $n=5$, we consider the pair $(M=M_5, \frac{1}{2}(H+H'))$, where $M=M_5$ and $H$, $H'$ are two distinct smooth hyperplane sections of $M$ with smooth intersection $I=H\cap H'$. So $H$ and $H'$ are $M_4$ and $I\cong M_3$. By Proposition~\ref{Proposition: degeneration to normal cone}, this pair can specially degenerate to $\big(\overline{\CC}(H, \CO_H(1)), \frac{1}{2}(H_\infty+\CE)\big)$, where $\CE=\overline{\CC}(I, \CO_I(1))$. Applying Proposition~\ref{K-semistability of cone} to $(H, \frac{1}{2}I)$ and $r=\frac{5}{2}$, since $\CO_H(1)\sim_\IQ -\frac{1}{r}(K_H+\frac{1}{2}I)$, we see that $\big(\overline{\CC}(H, \CO_H(1)), \frac{1}{2}(H_\infty+\CE)\big)$ is K-semistable. 
Then by the openness of K-semistability, $(M, \frac{1}{2}(H+H'))$ is K-semistable.  Similarly to the proof of the $n=4$ case, we can deform $(M, \frac{1}{2}(H+H'))$ to some $(M, \frac{1}{2}Q)$, and we get the K-semistability of $(M, \frac{1}{2}\tQ)$ for some smooth quadric divisor $\tQ$. 

We also consider the degeneration induced by some $Q\seq M_5$ which is an ordinary GM fourfold. This will be useful in the proof of K-stability. In the $n=4$ case, we already have a K-semistable special GM fourfold. By the openness of K-semistability, there exists a K-semistable ordinary GM fourfold $Q$. Since $-K_M\sim 2Q$, we have $r=1$ and $c_\max=\frac{4}{5}$. Then by Proposition~\ref{Proposition: degeneration to normal cone}, we see that $(M,\frac{4}{5}Q)$ is K-semistable. 
By interpolation we have that $(M, cQ)$ is K-semistable for any $\frac{1}{2}\le c\le \frac{4}{5}$.

For $n=6$, we already have a K-polystable base $M=M_6=\Gr(2,5)$. The same argument as in the previous paragraph shows that $(M,\frac{3}{4}Q)$ is K-semistable for some $Q$. By interpolation, we see that $(M, cQ)$ is K-semistable for any $0\le c\le \frac{3}{4}$. In particular, $(M, \frac{1}{2}Q)$ is K-semistable. 

Next, we prove that general special GM $n$-folds are K-stable for $n=4,5,6$. We give a proof by induction on $n$ using Theorem~\ref{Theorem: strengthen of cone}. 

Since any ordinary GM threefold $Q_3\seq M_4$ is K-stable and $\delta(M_4)=\frac{25}{27}>\frac{5}{12}$, we see that $(M_4, cQ_3)$ is K-stable for any $0\ll c < \frac{7}{8}$ by Remark~\ref{Remark. strengthen of cone}. By interpolation, $(M_4, cQ_3)$ is K-stable for any $\frac{1}{4}< c< \frac{7}{8}$ (this holds for any $\frac{1}{9}< c< \frac{7}{8}$ and general $Q_3\seq M_4$ by Theorem~\ref{Theorem: (M_4,(1/9)Q_0) K-polystable}). In particular, $(M_4, \frac{1}{2}Q_3)$ is K-stable, and $\Aut(M_4, Q_3)$ is finite. Let $X_4\to (M_4, \frac{1}{2}Q_3)$ be the double cover in Remark~\ref{Remark. structure of GM}\eqref{RsGM-2}. Then $\Aut(X_4)$ is finite by Lemma~\ref{Lemma. Aut(X_n) to Aut(M_n,Q_n-1) ker = Z/2 }. Hence $X_4$ is K-stable by \cite[Theorem 1.2]{LZ20}. 

By the openness of K-stability, we see that a general $Q_4\seq M_5$ is K-stable. Since $\delta(M_5)=\frac{15}{16}>\frac{3}{5}$, we see that $(M_5, cQ_4)$ is K-stable for any $0\ll c< \frac{4}{5}$ and general $Q_4$ by Remark~\ref{Remark. strengthen of cone}, hence for any $\frac{1}{8}< c < \frac{4}{5}$ and general $Q_4$ by Theorem~\ref{Theorem: (M_5,(1/8)Q_0) K-polystable} using interpolation. In particular, $(M_5, \frac{1}{2}Q_4)$ is K-stable and $\Aut(M_5, Q_4)$ is finite. Let $X_5\to (M_5, \frac{1}{2}Q_4)$ be the corresponding double cover. Then $\Aut(X_5)$ is finite by Lemma~\ref{Lemma. Aut(X_n) to Aut(M_n,Q_n-1) ker = Z/2 }, and $X_5$ is K-stable by \cite[Theorem 1.2]{LZ20}. 

Also by the openness of K-stability, we get a K-stable ordinary GM fivefold $Q_5 \seq M_6$. With the same argument, we conclude that $(M_6, cQ_5)$ is K-stable for any $0< c < \frac{3}{4}$. In particular, $(M_6, cQ_5)$ is K-stable with finite automorphism group. Hence its double cover $X_6$ has finite automorphism group by Lemma~\ref{Lemma. Aut(X_n) to Aut(M_n,Q_n-1) ker = Z/2 } and is K-stable by \cite[Theorem 1.2]{LZ20}. 

We conclude by the openness of K-stability that general special GM-manifolds are K-stable. 
\end{proof}

\renewcommand{\appendixname}{Appendix. Intersection theory of \texorpdfstring{$\boldsymbol{\IG(1,4)}$}{G(1,4)}}

\addappendix 

We use the notation of \cite{Ful98}. Let $\IG(d,n)$ be the Grassmannian parametrizing $d$-dimensional subspaces in $\IP^n$ (note that $\IG(d,n)\cong \Gr(d+1,n+1)$). For any flag $A_0\subset \cdots \subset A_d$ of subspaces in $\IP^n$ with $\dim A_i = a_i$, we have the following Schubert subvariety:  
\begin{eqnarray*} 
\Omega(A_0,\ldots, A_d) = \{L\in \IG(d,n)\mid \dim(L\cap A_i) \ge i, 0\le i\le d \}. 
\end{eqnarray*}
The Schubert cycle $[\Omega(A_0,\ldots, A_d)]$ depends only on $a_0,\ldots, a_d$ and is denoted by $(a_0,\ldots, a_d)$. One may also represent the Schubert cycles by Chern polynomials (see \cite[Section 14.7]{Ful98}) 
\begin{eqnarray*} 
\{\lam_0,\ldots, \lam_d\}=(a_0,\ldots, a_d), 
\end{eqnarray*} 
where $\lam_i=n-d+i-a_i$. 

Recall $V_5 = \spn\{e_0,\ldots, e_4\}$. Let $A_0=\IP(\spn\{e_0,e_1\})$ and $A_1=\IP(\spn\{e_0,e_1,e_2\})$; then 
$$S=\Omega(A_0,A_1)\seq \IG(1,4), $$ 
by Equation~(\ref{Eqnarray. def. S}), via the Pl\"ucker embedding $\IG(1,4)\seq \IP(\wedge^2V_5)$. Let $B_0=\IP(\spn\{e_4\})$ and $B_1=\IP(V_5)$; then 
$$W=\Omega(B_0,B_1)\seq \IG(1,4), $$ 
by Equation~(\ref{Eqnarray. def. W}). 
We see that $[S]=(1,2)=\{2,2\}$ and $[W]=(0,4)=\{3,0\}$ as Schubert cycles of $\IG(1,4)$. 

As in \cite[Example 14.7.2]{Ful98}, we denote the Schubert cycles of $\IG(1,4)$ by 
\begin{eqnarray*}
\left\{ \begin{array}{ccccc}
1 &=& (3,4) &=& \{0,0\},\\
\sigma_1 &=& (2,4) &=& \{1,0\},\\
\sigma_2 &=& (1,4) &=& \{2,0\}, \\
\sigma_3 &=& (0,4) &=& \{3,0\}, \\
\pt &=& (0,1) &=& \{3,3\}, \\
\end{array} \right. \quad
\left\{ \begin{array}{ccccc}
\chi_2 &=& (2,3) &=& \{1,1\},\\
\chi_3 &=& (1,3) &=& \{2,1\},\\
\chi_4 &=& (0,3) &=& \{3,1\},\\
\psi_4 &=& (1,2) &=& \{2,2\},\\
\psi_5 &=& (0,2) &=& \{3,2\}.
\end{array} \right. 
\end{eqnarray*}
We see that $\sigma_i=c_i(\CO_{\IG(1,4)}(1))$. 
By Giambelli's formula, see \cite[Proposition 14.6.4]{Ful98}, we have 
\begin{eqnarray*}
\left\{ \begin{array}{ccccc}
\chi_2 &=& \sigma_1^2-\sigma_2,\\
\chi_3 &=& \sigma_1\sigma_2-\sigma_3,\\
\chi_4 &=& \sigma_1\sigma_3,
\end{array} \right. \quad
\left\{ \begin{array}{ccccc}
\psi_4 &=& \sigma_1\sigma_3 - \sigma_2^2,\\
\psi_5 &=& \sigma_2\sigma_3,\\
\pt &=& \sigma_3^2. 
\end{array} \right. 
\end{eqnarray*}
By Pieri's formula, see \cite[Proposition 14.6.1]{Ful98}, we have a natural pairing 
\begin{eqnarray*}
\sigma_1\psi_5
= \sigma_2\chi_4
= \sigma_3^2
= \chi_3^2
= \chi_2\psi_4 = \pt, 
\end{eqnarray*}
and other pairings to the top degree are zero. Moreover, we have 
\begin{eqnarray*}
\left\{ \begin{array}{ccccc}
\sigma_1\chi_2 &=& \chi_3,\\
\sigma_1\chi_3 &=& \chi_4+\psi_4,\\
\sigma_1\chi_4 &=& \psi_5,\\
\sigma_1\psi_4 &=& \psi_5, 
\end{array} \right. \quad
\left\{ \begin{array}{ccccc}
\sigma_2\chi_2 &=& \chi_4,\\
\sigma_2\chi_3 &=& \psi_5,\\
\sigma_2\chi_4 &=& \pt,\\
\sigma_2\psi_4 &=& 0,\\
\end{array} \right. \quad
\left\{ \begin{array}{ccccc}
\sigma_1^2 &=& \sigma_2+\chi_2,\\
\sigma_1^3 &=& \sigma_3+2\chi_3,\\
\sigma_1^4 &=& 3\chi_4+2\psi_4,\\
\sigma_1^5 &=& 5\psi_5. 
\end{array} \right. 
\end{eqnarray*}

For any vector bundle $E$, we denote its total Chern class by $c(E)$ and its total Segre class by $s(E)$  (note that $c(E)\cdot s(E)=1$). 

\begin{thm}
\label{Theorem. c(N_S/M_4)}
The total Chern class of\, $N_{S/M_4}$ is 
$c(N_{S/M_4}) = 1+ 2c_1^2, $ where $c_1=c_1(\CO_S(1))$. 
\end{thm}

\begin{proof}
We simply set $\Gr=\Gr(2,V_5)=\IG(1,4)$. 
We have the following commutative diagram of short exact sequences of vector bundles on $S$: 
\begin{center}
\begin{tikzcd}[row sep=1.5em, column sep=1.5em]
  & & 0 \arrow[d] & 0 \arrow[d] & \\
  0 \arrow[r] & T_S \arrow[r] \arrow[d, equal] & T_{M_4}|_S \arrow[r] \arrow[d] & N_{S/M_4} \arrow[r] \arrow[d] & 0 \\
  0 \arrow[r] & T_S \arrow[r] & T_\Gr|_S \arrow[r] \arrow[d] & N_{S/\Gr} \arrow[r] \arrow[d] & 0\rlap{,} \\
  & & N_{M_4/\Gr}|_S \arrow[r, equal] \arrow[d] & N_{M_4/\Gr}|_S \arrow[d] & \\
  & & 0 & 0 &
\end{tikzcd}
\end{center}
where $N_{M_4/\Gr} \cong \CO_{M_4}(1)^{\oplus 2}$. Hence 
\begin{align*}
c(N_{S/M_4}) 
&= c(T_{M_4})|_S \cdot s(T_S) 
\,=\, \left(c(T_{\Gr})|_{M_4}\cdot s(N_{M_4/\Gr})\right)|_S \cdot s(T_S)\\
&= \left(c(T_{\Gr})\cdot (1+\sigma_1)^{-2}\right)|_S \cdot s(T_S). 
\end{align*}
By the Euler sequence of $\Gr$, we have the total Chern class of the tangent bundle $T_\Gr$: 
\begin{eqnarray*}
c(T_\Gr) 
=  1 + 5\sigma_1 + (12\sigma_1^2-\sigma_2) 
 + 15(2\sigma_1\sigma_2-\sigma_3) + 25\sigma_2^2 
 + 30\sigma_2\sigma_3 + 10\sigma_3^2. 
\end{eqnarray*}
We also have $c(T_S) = (1+c_1)^3$, where $c_1=c_1(\CO_S(1))$. Since $[S]=\psi_4$, $\sigma_1\psi_4=\psi_5$ and $\sigma_2\psi_4=0$, we have $\sigma_1|_S=c_1, \sigma_2|_S = 0$. Hence 
\begin{equation*}
c(N_{S/M_4}) 
= (1+5c_1+12c_1^2)(1-c_1+c_1^2)^2 \cdot (1-3c_1+6c_1^2)
= 1+2c_1^2.\qedhere
\end{equation*}
\end{proof}

\begin{thm}
\label{Theorem. c(N_W/M_5)}
The total Chern class of\, $N_{W/M_5}$ is 
$c(N_{W/M_5}) = 1+ c_1^2, $ where $c_1=c_1(\CO_W(1))$. 
\end{thm}

\begin{proof}
  Replacing $S$, $M_4$ by $W$, $M_5$, respectively, we have a similar diagram of short exact sequences to that in Theorem~\ref{Theorem. c(N_S/M_4)}. In this case, we have $N_{M_5/\Gr} \cong \CO_{M_5}(1)$ and $c(T_W) = (1+c_1)^4$, where $c_1=c_1(\CO_W(1))$. Since $[W]=\sigma_3$, $\sigma_1\sigma_3=\chi_4$, $\sigma_2\sigma_3=\psi_5$ and $\sigma_3^2=\pt$, we have $\sigma_i|_W=c_1^i$ for $i=1,2,3$. Hence
\begin{align*}
c(N_{W/M_5}) 
&= \left(c(T_{\Gr})|_{M_5}\cdot s(N_{M_5/\Gr})\right)|_W \cdot s(T_W)\\
&= (1 + 5c_1 + 11c_1^2 + 15c_1^3)
    (1 -  c_1 +   c_1^2 -   c_1^3) \cdot 
    (1 - 4c_1 + 10c_1^2 - 20c_1^3) 
\,=\, 1 + c_1^2.  \qedhere
\end{align*}
\end{proof}

\phantomsection 

\newcommand{\etalchar}[1]{$^{#1}$}
\providecommand{\bysame}{\leavevmode\hbox to3em{\hrulefill}\thinspace}

\end{document}